\documentclass[12pt]{article}
\usepackage{a4wide}
\usepackage{amsthm}
\usepackage{amsfonts}
\usepackage{amssymb}
\usepackage{amsmath}
\usepackage{cite}
\usepackage{epsfig}
\usepackage{tabularx}
\usepackage{multirow}

\newtheorem{theorem}{Theorem}

\newtheorem{lemma}[theorem]{Lemma}
\newtheorem{proposition}[theorem]{Proposition}

\newtheorem{problem}{Problem}

\newcommand\dd{\,\mbox{d}}
\newcommand\NN{{\mathbb N}}

\newcommand\RR{{\mathbb R}}

\DeclareTextCompositeCommand{\v}{OT1}{l}{l\nobreak\hspace{-.1em}'}
\DeclareTextCompositeCommand{\v}{OT1}{t}{t\nobreak\hspace{-.1em}'\nobreak\hspace{-.15em}}

\def\cepsfbox#1{$\vcenter{\hbox{\epsfbox{#1}}}$}

\begin{document}
\title{Quasirandom forcing orientations of cycles\thanks{The first author was supported by the National Science Centre grant number 2021/42/E/ST1/00193. The third and fourth authors were supported by the MUNI Award in Science and Humanities (MUNI/I/1677/2018) of the Grant Agency of Masaryk University.}}

\author{Andrzej Grzesik\thanks{Faculty of Mathematics and Computer Science, Jagiellonian University, {\L}ojasiewicza 6, 30-348 Krak\'{o}w, Poland. E-mail: {\tt Andrzej.Grzesik@uj.edu.pl}.}\and
        Daniel I\v{l}kovi\v{c}\thanks{Faculty of Informatics, Masaryk University, Botanick\'a 68A, 602 00 Brno, Czech Republic. E-mail: {\tt 493343@mail.muni.cz}.}\and
	Bart\l{}omiej Kielak\thanks{Faculty of Mathematics and Computer Science, Jagiellonian University, {\L}ojasiewicza 6, 30-348 Krak\'{o}w, Poland. E-mail: {\tt bartlomiej.kielak@doctoral.uj.edu.pl}.}\and
        Daniel Kr{\'a}\v{l}\thanks{Faculty of Informatics, Masaryk University, Botanick\'a 68A, 602 00 Brno, Czech Republic. E-mail: {\tt dkral@fi.muni.cz}.}}
\date{} 
\maketitle

\begin{abstract}
An oriented graph $H$ is quasirandom-forcing
if the limit (homomorphism) density of $H$ in a sequence of tournaments
is $2^{-\|H\|}$ if and only if the sequence is quasirandom.
We study generalizations of the following result:
the cyclic orientation of a cycle of length $\ell$ is quasirandom-forcing if and only if $\ell\equiv 2\mod 4$.

We show that no orientation of an odd cycle is quasirandom-forcing.
In the case of even cycles,
we find sufficient conditions on an orientation to be quasirandom-forcing,
which we complement by identifying necessary conditions.
Using our general results and spectral techniques used to obtain them,
we classify which orientations of cycles of length up to $10$ are quasirandom-forcing.
\end{abstract}

\section{Introduction}
\label{sec:intro}

The study of quasirandom structures,
i.e., structures with properties close to random structures,
is a classical topic in combinatorics.
Indeed, the notion of quasirandom graphs (as we understand it today)
appeared in the works of R\"odl~\cite{Rod86}, Thomason~\cite{Tho87,Tho87b} and
Chung, Graham and Wilson~\cite{ChuGW89} in the 1980s:
a sequence $(G_n)_{n\in\NN}$ of graphs is \emph{quasirandom}
if the density of any graph $H$ in the sequence converges to its expected density
in the Erd\H os-R\'enyi random graph with edge probability $1/2$.
This definition of graph quasirandomness turned out to be very robust as
it is equivalent to many additional properties that the Erd\H os-R\'enyi random graph has almost surely.
Somewhat surprisingly,
a sequence $(G_n)_{n\in\NN}$ of graphs is quasirandom if and only if
the densities of $K_2$ and $C_4$ converge to their expected densities.
Chung, Graham and Wilson~\cite{ChuGW89} showed that $C_4$ can be replaced with any even cycle, and
Skokan and Thoma~\cite{SkoT04} showed that $C_4$ can be replaced with any complete bipartite graph $K_{a,b}$ with $a,b\ge2$.
The Forcing Conjecture by Conlon, Fox and Sudakov~\cite{ConFS10},
a generalization of the famous Sidorenko's Conjecture,
asserts that $C_4$ can be replaced with any bipartite graph containing a cycle.
We also refer the reader to~\cite{ConHPS12,HubKPP19,ReiS19} for results concerning pairs of graphs
whose densities force quasirandomness but the pair does not contain $K_2$.

Quasirandomness has been studied in the setting of many different kinds of combinatorial structures,
in particular,
groups~\cite{Gow08},
hypergraphs~\cite{ChuG90,ChuG91s,Gow06,Gow07,HavT89,KohRS02,NagRS06,RodS04},
permutations~\cite{ChaKNPSV20,Coo04,KraP13,Kur22},
Latin squares~\cite{CooKLM22,EbeMM22,GarHHS20,GowL20},
subsets of integers~\cite{ChuG92}, etc.
Many of these notions are treated in a unified way in the recent paper by Coregliano and Razborov~\cite{CorR20}.
In the present paper,
we consider quasirandomness of tournaments as studied in~\cite{BucLSS19,ChuG91,CorPS19,CorR17,HanKKMPSV23}.

The starting point of our research is a recent solution of the conjecture of Bartley and Day~\cite{Bar18,Day17}
concerning the maximum density of a cyclically oriented cycle in a tournament.
The \emph{homomorphism density} of an oriented graph $H$ in an oriented graph $G$
is the probability that a random mapping from the vertices of $H$ to $G$ is a homomorphism;
note that the expected homomorphism density of a cyclically oriented cycle of length $\ell$ in a random tournament is $2^{-\ell}$.
For $\ell\ge 3$,
let $c(\ell)$ be the supremum over all $d$ such that there exist arbitrarily large tournaments
with homomorphism density of the cyclically oriented cycle of length $\ell$ at least $d\cdot 2^{-\ell}$.
Classical extremal combinatorics results~\cite{BeiH65,Col64,KenB40,KomM17,Moo15,Sze43} imply that $c(3)=c(5)=1$ and $c(4)=4/3$.
Bartley and Day~\cite{Bar18,Day17} conjectured that $c(\ell)=1$ if and only if $\ell$ is not divisible by four.
This conjecture was proven in~\cite{GrzKLV23} where it was also shown that
if $\ell\equiv 2\mod 4$, then quasirandom sequences of tournaments are the only sequences that
asymptotically maximize density of the cyclically oriented cycle of length~$\ell$,
while if $\ell\not\equiv 2\mod 4$, this property does not hold.

To state our results precisely, we need an additional definition.
We say that an oriented graph $H$ is \emph{quasirandom-forcing}
if a sequence of tournaments is quasirandom if and only if
the homomorphism density of $H$ in the tournaments in the sequence converges to $2^{-\|H\|}$,
where $\|H\|$ stands for the number of edges of $H$.
Among all tournaments, transitive tournaments with at least four vertices and
one exceptional five-vertex tournament are quasirandom-forcing~\cite{CorPS19,CorR17},
while all other tournaments are not quasirandom-forcing~\cite{BucLSS19,HanKKMPSV23}.
The above mentioned result from~\cite{GrzKLV23} says that
a cyclically oriented cycle of length $\ell$ is quasirandom-forcing if and only if $\ell\equiv 2\mod 4$.

Motivated by this result, we address the following question:
\emph{What orientations of a cycle of length $\ell$ are quasirandom-forcing?}.
We first show that no orientation of an odd cycle is quasirandom-forcing (Theorem~\ref{thm:odd}).
In the case of even cycles,
we find general conditions on an orientation that imply that the orientation is quasirandom-forcing, and
also present necessary conditions.
We use the identified conditions to provide a full classification
which orientations of cycles of length four, six, eight and ten are quasirandom-forcing.
While we originally hoped that we will be able to obtain a full classification for cycles of all lengths,
e.g., by considering the number of forward and backward arcs or
through the existence of homomorphisms to a small set of oriented graphs,
it turned out that the boundary between orientations that are and that are not quasirandom-forcing is rather mysterious;
see Section~\ref{sec:concl} for detailed discussion.

The paper is structured as follows.
In Section~\ref{sec:prelim}, we introduce notation used throughout the paper,
including the notions and tools from the theory of combinatorial limits that we use to treat the problem;
we also establish several auxiliary claims related to tournament limits and kernels.
In Section~\ref{sec:general},
we present general results guaranteeing that a particular orientation of a cycle
is quasirandom-forcing, which apply to orientations of cycles of arbitrary length.
We complement these results in Section~\ref{sec:negative} by presenting methods to show that
a particular orientation of a cycle is not quasirandom-forcing.
The general results of Sections~\ref{sec:general} and~\ref{sec:negative} are strong enough
to yield a full classification of orientations of the cycles of length four, six and eight, and
a classification of all but four orientations of the cycle of length ten.
We present the classification results implied by our general results and
the analysis of the four uncovered orientations of the cycle of length ten in Section~\ref{sec:specific}.
Finally,
we conclude in Section~\ref{sec:concl} by discussing prospects of a classification of orientations of cycles of all lengths.

\section{Preliminaries}
\label{sec:prelim}

We treat the problems addressed in this paper using the language and tools from the theory of combinatorial limits.
We next provide a brief introduction to tournament limits and concepts related to them.
We refer the reader to the monograph by Lov\'asz~\cite{Lov12} for general introduction to theory of graph limits, and
to~\cite{ChaGKN19,GrzKLV23,Tho18,ZhaZ20} specifically for limits of oriented graphs and tools employing them.

We start with fixing general notation.
If $H$ is an oriented graph, then $V(H)$ is the vertex set of $H$ and $E(H)$ is the edge set of $H$;
the number of vertices of $H$ is denoted by $|H|$ and the number of edges by $\|H\|$.
If $H$ is an oriented graph and $F$ is a subset of $E(H)$,
then $H\langle F\rangle$ is the spanning subgraph of $H$ with edge set $F$,
i.e., the vertex set of $H\langle F\rangle$ is $V(H)$ and the edge set is $F$.
An oriented graph is a \emph{tournament}
if every pair of vertices of $H$ is joined by exactly one edge (oriented in one of the two possible directions).
Two specific orientations of paths and cycles with even edges will play an important role in our arguments:
for a positive integer $k$,
$Q_{2k}$ is the orientation of the $2k$-edge path such that
all edges are oriented from the central vertex of the path towards its two end vertices, and
$D_{2k}$ is the orientation of the $2k$-vertex cycle obtained by choosing two opposite vertices and
orienting all edges in the direction from one of them to the other.
See Figure~\ref{fig:QD} for the notation.

\begin{figure}
\begin{center}
\epsfbox{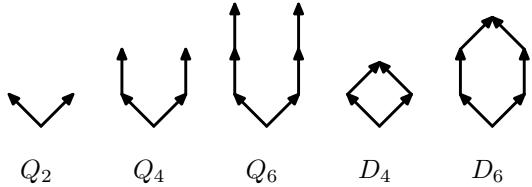}
\end{center}
\caption{The oriented graphs $Q_2$, $Q_4$, $Q_6$, $D_4$ and $D_6$.}
\label{fig:QD}
\end{figure}

As we have already defined in Section~\ref{sec:intro},
a \emph{homomorphism density} of an oriented graph~$H$ in an oriented graph $G$, denoted by $t(H,G)$
is the probability that a random mapping from the vertices of $H$ to $G$ is a homomorphism.
We say that a sequence of oriented graphs $(G_n)_{n\in\NN}$ is \emph{convergent}
if the sequence $t(H,G_n)$ converges for every oriented graph~$H$.
A~\emph{kernel} is a bounded measurable function from $[0,1]^2$ to $\RR$, and
a \emph{tournamenton} is a non-negative kernel $W$ such that $W(x,y)+W(y,x)=1$ for all $(x,y)\in [0,1]^2$.
We slightly abuse the notation by writing $p$ for $p\in\RR$ for the kernel equal to $p$ everywhere.
If $U$ and $U'$ are two kernels, we write $U\equiv U'$ if $U$ and $U'$ are equal almost everywhere.
A~kernel~$U$ is \emph{antisymmetric} if $U(x,y)=-U(y,x)$ for all $(x,y)\in [0,1]^2$.
In particular, if $W$ is a tournamenton, then $W-1/2$ is an antisymmetric kernel.
We often view a kernel~$U$ as an operator on $L_2[0,1]$ defined as
\[(Uf)(x)=\int_{[0,1]}U(x,y)f(y)\dd y.\]
We write $U^k$ for the composition of $U$ with itself $k$ times.
For example,
\[U^2(x,y)=\int_{[0,1]}U(x,z)U(z,y)\dd z.\]

We next extend the definition of a homomorphism density from oriented graphs to kernels.
The \emph{homomorphism density} of an oriented graph $H$ in a kernel $U$ as
\[t(H,U)=\int_{[0,1]^{V(H)}}\prod_{vw\in E(H)}U(x_v,x_w)\dd x_{V(H)}.\]
If $(G_n)_{n\in\NN}$ is a convergent sequence of tournaments,
there exists a tournamenton $W$ such that $t(H,W)$ is the limit density of $t(H,G_n)$ for every oriented graph $H$, and vice versa
for every tournamenton $W$, there exists a convergent sequence $(G_n)_{n\in\NN}$ of tournaments such that
$t(H,G_n)$ converges to $t(H,W)$ for every oriented graph $H$.

The definition of a quasirandom-forcing graph using the theory of combinatorial limits can be cast as follows:
an oriented graph $H$ is \emph{quasirandom-forcing}
if every tournamenton~$W$ such that $t(H,W)=2^{-\|H\|}$ satisfies that $W\equiv 1/2$.
As we have already presented, a cyclically oriented cycle of length $\ell$
is quasirandom-forcing if and only if $\ell\equiv 2\mod 4$.

\subsection{Properties of antisymmetric kernels}

We now review some basic properties of antisymmetric kernels.
The next proposition directly follows from the definition of a density of an oriented graph in a kernel.

\begin{proposition}
\label{prop:reorient}
Let $H$ be an oriented graph and
let $H'$ be an oriented graph obtained by reversing the orientation of exactly one edge of $H$.
Every antisymmetric kernel $U$ satisfies that $t(H,U)=-t(H',U)$.
\end{proposition}

Proposition~\ref{prop:reorient} yields the following.

\begin{proposition}
\label{prop:odd}
Let $H$ be an orientation of a path or a cycle with an odd number of edges.
Every antisymmetric kernel $U$ satisfies that $t(H,U)=0$.
\end{proposition}

\begin{proof}
Fix an antisymmetric kernel $U$.
Consider a directed path $P$ (an oriented path with all edges oriented forward) with an odd number of edges.
Since reversing the orientation of all edges of $P$ results in a directed path,
we obtain using Proposition~\ref{prop:reorient} by reversing one edge of $P$ after another that $t(P,U)=-t(P,U)$.
Hence, it holds that $t(P,U)=0$ and so
the homomorphism density of any orientation of a path with an odd number of edges (as it can be obtained by reversing some edges of $P$)
in the kernel $U$ is zero by Proposition~\ref{prop:reorient}.

The argument for orientations of a cycle is analogous.
Consider a cyclically oriented cycle $C$ with an odd number of edges.
Since reversing the orientation of all edges of $C$ results in a cyclically oriented cycle,
we obtain $t(C,U)=-t(C,U)$ by Proposition~\ref{prop:reorient} and so $t(C,U)=0$.
Hence, the homomorphism density of any orientation of a cycle with an odd number of edges in $U$ is zero.
\end{proof}

Recall that a kernel $U$ can be viewed as a linear operator on $L_2[0,1]$.
The next bound on spectral radius of the operator $U^2$ for an antisymmetric kernel with $\|U\|_{\infty}\le 1/2$
is implied by~\cite[Lemma 12]{GrzKLV23};
we remark that
the bound is tight as
it is attained for the antisymmetric kernel $U$ such that $U(x,y)=1/2$ for $x<y$ and $U(x,y)=-1/2$ for $x>y$.

\begin{proposition}
\label{prop:radius}
Let $U:[0,1]^2\to [-1/2,1/2]$ be an antisymmetric kernel.
The spectral radius of the operator $U^2$ is at most $1/\pi^2$.
\end{proposition}

Using Proposition~\ref{prop:radius},
we bound the density of an oriented graph $Q_{2m}$ using the density of $Q_{2k}$ for $k<m$ in an antisymmetric kernel as follows.

\begin{lemma}
\label{lm:positive}
For every $k\in\NN$ and every antisymmetric kernel $U:[0,1]^2\to [-1/2,1/2]$,
it holds that $t(Q_{2k},U)\ge 0$.
In addition, it holds that
\[t(Q_{2m},U)\le \frac{t(Q_{2k},U)}{\pi^{2(m-k)}}\]
for every $m\ge k+1$.
\end{lemma}

\begin{proof}
Fix an antisymmetric kernel $U:[0,1]^2\to [-1/2,1/2]$ and an integer $k\in\NN$.
Let $h_1:[0,1]\to\RR$ be the function equal to $1$ for all $x\in [0,1]$.
Observe that
\[t(Q_{2k},U)=\int_{[0,1]} (U^{k}h_1)(x)^2 \dd x=\langle U^{k}h_1,U^{k}h_1\rangle,\]
which implies that $t(Q_{2k},U)\ge 0$.
Let $\rho$ be the spectral radius of $U^2$; note that $\rho\le 1/\pi^2$ by Proposition~\ref{prop:radius}.
Next consider $m\ge k+1$. Since $U$ is antisymmetric and so $U^2$ is symmetric,
we obtain that
\begin{align*}
 t(Q_{2m},U)&=\langle U^{m}h_1,U^{m}h_1\rangle
              =\left|\langle U^{k}h_1,U^{2m-k}h_1\rangle\right|\\
	     &\le \rho^{m-k}\langle U^{k}h_1,U^{k}h_1\rangle=\rho^{m-k}t(Q_{2k},U),
\end{align*}
which yields the second part of the statement of the lemma.
\end{proof}

Proposition~\ref{prop:radius} and Lemma~\ref{lm:positive} yield the following.

\begin{lemma}
\label{lm:polynomial}
Let $k\in\NN$ and let $\alpha_1,\ldots,\alpha_k$ be reals such that
\[\sum_{j=1}^k\alpha_j x^j\ge 0\]
for every $x\in [0,1/\pi^2]$.
Every antisymmetric kernel $U:[0,1]^2\to [-1/2,1/2]$ satisfies that
\[\sum_{j=1}^k\alpha_j t(Q_{2j},U)\ge 0.\]
\end{lemma}

\begin{proof}
We first observe that $-U^2$ is a positive semidefinite operator:
indeed, consider $f\in L_2[0,1]$ and note that
\[\langle f,-U^2 f\rangle=\langle Uf,Uf\rangle=\|Uf\|_2^2\ge 0.\]
Since $-U^2$ is a Hilbert–Schmidt operator,
there exists a finite or countably infinite set $I$ and non-zero reals $\lambda_i$, $i\in I$, such that
$\lambda_i$, $i\in I$, are all non-zero eigenvalues of $-U^2$.
Since $-U^2$ is positive semidefinite, it holds that $\lambda_i>0$ for every $i\in I$.
Let $f_i$, $i\in I$, be the corresponding eigenfunctions;
we can assume that $f_i$ are orthonormal.

As in the proof of Lemma~\ref{lm:positive}, 
let $h_1:[0,1]\to\RR$ be the function equal to $1$ for all $x\in [0,1]$.
We next define
\[\sigma_i=\int_{[0,1]}f_i(x)\dd x=\langle f_i,h_1\rangle,\]
and observe that
\[t(Q_{2j},U)=\langle U^jh_1,U^jh_1\rangle=\langle h_1,(-U^2)^jh_1\rangle=\sum_{i\in I}\lambda_i^j\sigma_i^2\]
for every $j\in\NN$.
Since the spectral radius of $-U^2$ is at most $1/\pi^2$ by Proposition~\ref{prop:radius},
we obtain from $0\le\lambda_i\le 1/\pi^2$ that
\[\sum_{j=1}^k\alpha_j\lambda_i^j\ge 0\]
for every $i\in I$.
It follows that
\[\sum_{j=1}^k\alpha_j t(Q_{2j},U)=\sum_{i\in I}\sigma_i^2\sum_{j=1}^k\alpha_j\lambda_i^j\ge 0,\]
as wanted.
\end{proof}

We finish this section with a lemma concerning the density of a cycle $D_{2k}$ in an antisymmetric kernel.

\begin{lemma}
\label{lm:cycle}
Let $k\ge 2$.
Every antisymmetric kernel $U$ satisfies that $t(D_{2k},U)\ge 0$ and
the equality holds if and only if $U\equiv 0$.
\end{lemma}

\begin{proof}
Fix $k\ge 2$ and an antisymmetric kernel $U$.
Let $C$ be the cycle with $2k$ edges whose orientation alternate, and
note that $t(C,U)=t(D_{2k},U)$ by Proposition~\ref{prop:reorient}.
Observe that
\begin{equation}
t(C,U)=\int_{[0,1]^2}(-U^2)(x,y)(-U^2)^{k-1}(x,y)\dd x\dd y.\label{eq:cycle1}
\end{equation}
Similarly to the proof of Lemma~\ref{lm:polynomial},
let $\lambda_i$, for a finite or countably infinite set $I$, be all non-zero eigenvalues of $-U^2$, and
let $f_i$ be the corresponding orthonormal eigenfunctions.
Note that
\begin{equation}
(-U^2)(x,y)=\sum_{i\in I}\lambda_i f_i(x)f_i(y).\label{eq:cycle2}
\end{equation}
Using \eqref{eq:cycle1} and \eqref{eq:cycle2}, we obtain that
\[t(C,U)=\sum_{i\in I}\lambda_i^k.\]
Since the operator $-U^2$ is positive semidefinite,
it holds that $\lambda_i>0$ for every $i\in I$.
Hence, it holds that $t(D_{2k},U)=t(C,U)\ge 0$ and
the equality holds if and only if $I=\emptyset$, i.e., $U\equiv 0$.
\end{proof}

\section{General analysis}
\label{sec:general}

In this section,
we present general results that guarantee a given orientation of a cycle to be quasirandom-forcing.
We start with presenting an identity that is crucial in our arguments.
It describes the density of an oriented cycle $C$ in a tournamenton $W=1/2+U$ in terms of a deviation $U$ from the ``quasirandom'' tournamenton $1/2$, and
it can be straightforwardly derived from the definition of $t(C,W)$, see e.g.~\cite[Proposition~16.27]{Lov12}.
\begin{equation}
t\left(C,\frac{1}{2}+U\right)
 =\!\!\!\sum_{F\subseteq E(C)}\!t\left(C\langle E(C)\setminus F\rangle,\frac{1}{2}\right) t\left(C\langle F\rangle,U\right)
 =\!\!\!\sum_{F\subseteq E(C)}\frac{1}{2^{\|C\|-|F|}}\cdot t\left(C\langle F\rangle,U\right).
\label{eq:expand}
\end{equation}

Since $U$ is an antisymmetric kernel, Proposition~\ref{prop:odd} implies that
if $F\subseteq E(C)$ has an odd number of elements,
then $t\left(C\langle F\rangle,U\right)$ is zero.
Thus, the only non-zero terms in the sum are those with $F$ having an even number of elements.
More generally, if $C\langle F\rangle$ has a component with an odd number of edges,
then $t\left(C\langle F\rangle,U\right)=0$.

Therefore, we need to analyze the terms such that
each component of $C\langle F\rangle$ has an even number of edges.
This leads us to define the following quantities. 
For an oriented cycle~$C$ and positive integers $n_1,\ldots,n_k$,
we set $\alpha_C(2n_1,\ldots,2n_k)$ to be the difference between
\begin{itemize}
\item the number of sets $F\subseteq E(C)$ such that $C\langle F\rangle$ can be obtained from $Q_{2n_1}\cup\cdots\cup Q_{2n_k}$ by reversing an even number of edges and possibly adding some isolated vertices and
\item the number of sets $F\subseteq E(C)$ such that $C\langle F\rangle$ can be obtained from $Q_{2n_1}\cup\cdots\cup Q_{2n_k}$ by reversing an odd number of edges and possibly adding some isolated vertices.
\end{itemize}
In addition, if the length of the cycle $C$ is even, say $\ell$,
we set $\gamma_C$ to $+1$ if $C$ can be obtained from $D_{\ell}$ by reversing an even number of edges, and to $-1$ otherwise.
The just defined quantities $\alpha_C(2n_1,\ldots,2n_k)$ and $\gamma_C$
for all non-isomorphic orientations of cycles of length $4$, $6$ and $8$
can be found in Tables~\ref{tab:4}, \ref{tab:6} and \ref{tab:8}, respectively.

Plugging the just introduced notation in \eqref{eq:expand},
we obtain that the following holds for any orientation $C$ of a cycle of \emph{odd} length $\ell$:
\begin{equation}
t\left(C,\frac{1}{2}+U\right)=\frac{1}{2^{\ell}}+
                              \sum_{\substack{1\le n_1\le\cdots\le n_k\\ k+2n_1+\cdots+2n_k\le\ell}}\frac{\alpha_C(2n_1,\ldots,2n_k)}{2^{\ell-2n_1-\cdots-2n_k}}\prod_{i=1}^k t(Q_{2n_i},U)
\label{eq:odd}
\end{equation}
as $t(Q_{2n_1}\cup\ldots\cup Q_{2n_k},U)=t(Q_{2n_1},U)t(Q_{2n_2},U)\cdots t(Q_{2n_k},U)$.
Similarly, the following holds for any orientation $C$ of a cycle of \emph{even} length $\ell$:
\begin{equation}
t\left(C,\frac{1}{2}+U\right)=\frac{1}{2^{\ell}}+
                              \sum_{\substack{1\le n_1\le\cdots\le n_k\\ k+2n_1+\cdots+2n_k\le\ell}}\frac{\alpha_C(2n_1,\ldots,2n_k)}{2^{\ell-2n_1-\cdots-2n_k}}\prod_{i=1}^k t(Q_{2n_i},U)+\gamma_C t(D_{\ell},U).
\label{eq:even}
\end{equation}

In the rest of the section, we use \eqref{eq:even} and spectral arguments
to identify sufficient conditions on an orientation $C$ of a cycle of even length
to be quasirandom-forcing.
The simplest of them is given in the next theorem.

\begin{theorem}
\label{thm:QRsigns}
Let $C$ be an orientation of an even cycle.
If $\gamma_C=1$ and $\alpha_C(2n_1,\ldots,2n_k)$ is non-negative for all positive integers $n_1,\ldots,n_k$,
then $C$ is quasirandom-forcing.

Similarly, if $\gamma_C=-1$ and $\alpha_C(2n_1,\ldots,2n_k)$ is non-positive for all positive integers $n_1,\ldots,n_k$,
then $C$ is quasirandom-forcing.
\end{theorem}

\begin{proof}
Fix an orientation $C$ of a cycle of even length $\ell$ that satisfies the assumption of the theorem.
We present the proof when $\gamma_C=1$ as an argument in the other case is completely symmetric.
Consider a tournamenton $W$ and let $U=W-1/2$.
Since it holds that
$\alpha_C(2n_1,\ldots,2n_k)\ge 0$ for all positive integers $n_1,\ldots,n_k$ by the assumption of the theorem and
$t(Q_{2m},U)\ge 0$ for all positive integers $m$ by Lemma~\ref{lm:positive},
the inequality \eqref{eq:even} simplifies to
\[t(C,W)=t\left(C,\frac{1}{2}+U\right) \ge \frac{1}{2^{\ell}}+t(D_{\ell},U).\]
Hence, $t(C,W)=1/2^{\ell}$ only if $t(D_\ell,U)=0$.
As Lemma~\ref{lm:cycle} yields that $t(D_\ell,U)=0$ if only if $U\equiv 0$,
we obtain that $t(C,W)=1/2^{\ell}$ if only if $W\equiv 1/2$.
\end{proof}

For the additional sufficient conditions on an orientation of a cycle to be quasirandom-forcing,
we need the following estimate on the density of $Q_{2m}$ in an antisymmetric kernel.

\begin{lemma}
\label{lm:Q2}
Every antisymmetric kernel $U:[0,1]^2\to [-1/2,1/2]$ satisfies $t(Q_2,U)\le 1/12$.
More generally, it holds that
\[t(Q_{2m},U)\le\frac{1}{12\pi^{2(m-1)}}\]
for every positive integer $m$.
\end{lemma}

\begin{proof}
Fix an antisymmetric kernel $U:[0,1]^2\to [-1/2,1/2]$.
By Lemma~\ref{lm:positive}, it is enough to show that $t(Q_2,U)\le 1/12$.
Let $T_3$ be the transitive orientation of the cycle of length three.
The identity \eqref{eq:odd} implies that
\[t(T_3,1/2+U)=1/8+t(Q_2,U)/2.\]
Since $t(T_3,H)\le 1/6$ for every tournament $H$ (the equality is asymptotically attained
if every triple of vertices of $H$ induces a copy of $T_3$),
it holds that $t(T_3,W)\le 1/6$ for every tournamenton $W$.
In particular, it holds that $t(T_3,1/2+U)\le 1/6$,
which implies that $t(Q_2,U)/2\le 1/6-1/8=1/24$.
This yields that $t(Q_2,U)\le 1/12$ as desired.
\end{proof}

We are now ready to prove our second sufficient condition on an orientation of a cycle to be quasirandom-forcing.

\begin{theorem}
\label{thm:QRgeneral}
Let $C$ be an orientation of a cycle of even length $\ell$.
The oriented cycle $C$ is quasirandom-forcing if
\[\gamma_C=1\quad\mbox{and}\quad
  \frac{\alpha_C(2)}{2^{\ell-2}}+\sum_{\substack{1\le n_1\le\cdots\le n_k\\ 3<k+2n_1+\cdots+2n_k\le\ell}}\frac{\min\{0,\alpha_C(2n_1,\ldots,2n_k)\}}{12^{k-1}\cdot\pi^{2n_1+\cdots+2n_k-2k}\cdot 2^{\ell-2n_1-\cdots-2n_k}}\ge 0.\]
  
Similarly, the oriented cycle $C$ is quasirandom-forcing if
\[\gamma_C=-1\quad\mbox{and}\quad
  \frac{\alpha_C(2)}{2^{\ell-2}}+\sum_{\substack{1\le n_1\le\cdots\le n_k\\ 3<k+2n_1+\cdots+2n_k\le\ell}}\frac{\max\{0,\alpha_C(2n_1,\ldots,2n_k)\}}{12^{k-1}\cdot\pi^{2n_1+\cdots+2n_k-2k}\cdot 2^{\ell-2n_1-\cdots-2n_k}}\le 0.\]
\end{theorem}

\begin{proof}
Fix an orientation $C$ of a cycle of even length $\ell$ that satisfies the assumption of the theorem.
We present the proof when $\gamma_C=1$ as an argument in the other case is completely symmetric.
Consider a tournamenton $W$ and let $U=W-1/2$.
We use \eqref{eq:even}, Lemmas~\ref{lm:positive} and~\ref{lm:Q2} and
the assumption of the theorem to derive the following series of inequalities:

\begin{align*}
&t\left(C,\frac{1}{2}+U\right)=\frac{1}{2^{\ell}}+
                              \sum_{\substack{1\le n_1\le\cdots\le n_k\\ k+2n_1+\cdots+2n_k\le\ell}}\frac{\alpha_C(2n_1,\ldots,2n_k)}{2^{\ell-2n_1-\cdots-2n_k}}\prod_{i=1}^k t(Q_{2n_i},U)+t(D_{\ell},U)\\
			     &\ge\frac{1}{2^{\ell}}+\frac{\alpha_C(2)t(Q_2,U)}{2^{\ell-2}}+t(D_{\ell},U)
                             +\!\!\!\sum_{\substack{1\le n_1\le\cdots\le n_k\\ 3<k+2n_1+\cdots+2n_k\le\ell}}\!\!\!\frac{\min\{0,\alpha_C(2n_1,\ldots,2n_k)\}}{2^{\ell-2n_1-\cdots-2n_k}}\prod_{i=1}^k t(Q_{2n_i},U)\\
			     &\ge\frac{1}{2^{\ell}}+\frac{\alpha_C(2)t(Q_2,U)}{2^{\ell-2}}+t(D_{\ell},U)
                             +\sum_{\substack{1\le n_1\le\cdots\le n_k\\ 3<k+2n_1+\cdots+2n_k\le\ell}}\frac{\min\{0,\alpha_C(2n_1,\ldots,2n_k)\}t(Q_2,U)}{12^{k-1}\cdot\pi^{2n_1+\cdots+2n_k-2k}\cdot 2^{\ell-2n_1-\cdots-2n_k}}\\
			     &\ge\frac{1}{2^{\ell}}+t(D_{\ell},U)
			     +\left(\frac{\alpha_C(2)}{2^{\ell-2}}
                             +\hspace{-0.7em}\sum_{\substack{1\le n_1\le\cdots\le n_k\\ 3<k+2n_1+\cdots+2n_k\le\ell}}\hspace{-0.6em}\frac{\min\{0,\alpha_C(2n_1,\ldots,2n_k)\}}{12^{k-1}\cdot\pi^{2n_1+\cdots+2n_k-2k}\cdot 2^{\ell-2n_1-\cdots-2n_k}}\right)t(Q_2,U)\\
			     &\ge\frac{1}{2^{\ell}}+t(D_{\ell},U).
\end{align*}
Hence, $t(C,W)=1/2^{\ell}$ can hold only if $t(D_\ell,U)=0$.
Since $t(D_\ell,U)=0$ holds if and only if $W\equiv 1/2$ (or equivalently $U\equiv 0$) by Lemma~\ref{lm:cycle},
we obtain that the cycle $C$ is quasirandom-forcing.
\end{proof}

We next prove our third sufficient condition on an orientation of a cycle to be quasi\-random-forcing.

\begin{theorem}
\label{thm:QR22general}
Let $C$ be an orientation of a cycle of even length $\ell$.
The oriented cycle $C$ is quasirandom-forcing if $\gamma_C=1$,
\begin{align*}
\frac{\alpha_C(2)}{2^{\ell-2}}+\sum_{\substack{n_1=2,\ldots,\ell/2-1}}\frac{\min\{0,\alpha_C(2n_1)\}}{\pi^{2n_1-2}\cdot 2^{\ell-2n_1}} & \ge 0 &\mbox{and} \\
\frac{\alpha_C(2,2)}{2^{\ell-4}}+\sum_{\substack{1\le n_1\le\cdots\le n_k,\;2\le k\\ 6<k+2n_1+\cdots+2n_k\le\ell}}\frac{\min\{0,\alpha_C(2n_1,\ldots,2n_k)\}}{12^{k-2}\cdot\pi^{2n_1+\cdots+2n_k-2k}\cdot 2^{\ell-2n_1-\cdots-2n_k}} & \ge 0.
\end{align*}

Similarly, the oriented cycle $C$ is quasirandom-forcing if $\gamma_C=-1$,
\begin{align*}
\frac{\alpha_C(2)}{2^{\ell-2}}+\sum_{\substack{n_1=2,\ldots,\ell/2-1}}\frac{\max\{0,\alpha_C(2n_1)\}}{\pi^{2n_1-2}\cdot 2^{\ell-2n_1}} & \le 0 &\mbox{and} \\
\frac{\alpha_C(2,2)}{2^{\ell-4}}+\sum_{\substack{1\le n_1\le\cdots\le n_k,\;2\le k\\ 6<k+2n_1+\cdots+2n_k\le\ell}}\frac{\max\{0,\alpha_C(2n_1,\ldots,2n_k)\}}{12^{k-2}\cdot\pi^{2n_1+\cdots+2n_k-2k}\cdot 2^{\ell-2n_1-\cdots-2n_k}} & \le 0.
\end{align*}
\end{theorem}

\begin{proof}
Fix an orientation $C$ of a cycle of even length $\ell$ that satisfies the assumption of the theorem.
As in the proof of Theorem~\ref{thm:QRgeneral},
we present an argument for the case $\gamma_C=1$ as the other case is completely symmetric.
Consider a tournamenton $W$ and let $U=W-1/2$.
As in the proof of Theorem~\ref{thm:QRgeneral},
we again use \eqref{eq:even} and Lemmas~\ref{lm:positive} and~\ref{lm:Q2}
but we bound the terms with $k=1$ and those with $k\ge 2$ separately.
The terms with $k=1$ can be bound using Lemma~\ref{lm:positive} as follows:
\begin{align*}
\sum_{\substack{n_1=1,\ldots,\ell/2-1}}\frac{\alpha_C(2n_1)t(Q_{2n_1},U)}{2^{\ell-2n_1}}
& \ge \frac{\alpha_C(2)t(Q_2,U)}{2^{\ell-2}}+\sum_{\substack{n_1=2,\ldots,\ell/2-1}}\frac{\min\{0,\alpha_C(2n_1)\}t(Q_{2n_1},U)}{2^{\ell-2n_1}} \\
& \ge \frac{\alpha_C(2)t(Q_2,U)}{2^{\ell-2}}+\sum_{\substack{n_1=2,\ldots,\ell/2-1}}\frac{\min\{0,\alpha_C(2n_1)\}t(Q_2,U)}{\pi^{2n_1-2}\cdot 2^{\ell-2n_1}} \\
& = \left(\frac{\alpha_C(2)}{2^{\ell-2}}+\sum_{\substack{n_1=2,\ldots,\ell/2-1}}\frac{\min\{0,\alpha_C(2n_1)\}}{\pi^{2n_1-2}\cdot 2^{\ell-2n_1}}\right)t(Q_2,U) \ge 0.
\end{align*}
The remaining terms are bounded using Lemmas~\ref{lm:positive} and~\ref{lm:Q2} as follows:
\begin{align*}
& \sum_{\substack{1\le n_1\le\cdots\le n_k,\;2\le k\\ 5< k+2n_1+\cdots+2n_k\le\ell}} \frac{\alpha_C(2n_1,\ldots,2n_k)}{2^{\ell-2n_1-\cdots-2n_k}}\prod_{i=1}^k t(Q_{2n_i},U)\\
& \ge \frac{\alpha_C(2,2)t(Q_2,U)^2}{2^{\ell-4}}+\sum_{\substack{1\le n_1\le\cdots\le n_k,\;2\le k\\ 6<k+2n_1+\cdots+2n_k\le\ell}}\frac{\min\{0,\alpha_C(2n_1,\ldots,2n_k)\}}{2^{\ell-2n_1-\cdots-2n_k}}\prod_{i=1}^k t(Q_{2n_i},U)\\
& \ge \frac{\alpha_C(2,2)t(Q_2,U)^2}{2^{\ell-4}}+\sum_{\substack{1\le n_1\le\cdots\le n_k,\;2\le k\\ 6<k+2n_1+\cdots+2n_k\le\ell}}\frac{\min\{0,\alpha_C(2n_1,\ldots,2n_k)\}}{12^{k-2}\cdot\pi^{2n_1+\cdots+2n_k-2k}\cdot 2^{\ell-2n_1-\cdots-2n_k}}t(Q_2,U)^2\ge 0.
\end{align*}
The above two estimates yield that the value of the sum in \eqref{eq:even} is always non-negative and
so we obtain that
\[t\left(C,\frac{1}{2}+U\right)\ge\frac{1}{2^{\ell}}+t(D_{\ell},U).\]
Hence, $t(C,W)=1/2^{\ell}$ can hold only if $t(D_\ell,U)=0$.
Since $t(D_\ell,U)=0$ holds if and only if $W\equiv 1/2$ (or equivalently $U\equiv 0$) by Lemma~\ref{lm:cycle},
we obtain that the cycle $C$ is quasirandom-forcing.
\end{proof}

We next prove our fourth sufficient condition on an orientation of a cycle to be quasi\-random-forcing.

\begin{theorem}
\label{thm:QR4general}
Let $C$ be an orientation of a cycle of even length $\ell$.
The oriented cycle $C$ is quasirandom-forcing if $\gamma_C=1$,
\begin{align*}
\frac{\alpha_C(2)}{2^{\ell-2}}+\sum_{\substack{n_1=\cdots=n_k=1\\ 3<k+2n_1+\cdots+2n_k\le\ell}}\frac{\min\{0,\alpha_C(2n_1,\ldots,2n_k)\}}{12^{k-1}\cdot 2^{\ell-2n_1-\cdots-2n_k}} & \ge 0 &\mbox{and} \\
\frac{\alpha_C(4)}{2^{\ell-4}}+\sum_{\substack{1\le n_1\le\cdots\le n_k,\;2\le n_k\\ 5<k+2n_1+\cdots+2n_k\le\ell}}\frac{\min\{0,\alpha_C(2n_1,\ldots,2n_k)\}}{12^{k-1}\cdot\pi^{2n_1+\cdots+2n_k-2k-2}\cdot 2^{\ell-2n_1-\cdots-2n_k}} & \ge 0.
\end{align*}
Similarly, the oriented cycle $C$ is quasirandom-forcing if $\gamma_C=-1$,
\begin{align*}
\frac{\alpha_C(2)}{2^{\ell-2}}+\sum_{\substack{n_1=\cdots=n_k=1\\ 3<k+2n_1+\cdots+2n_k\le\ell}}\frac{\max\{0,\alpha_C(2n_1,\ldots,2n_k)\}}{12^{k-1}\cdot 2^{\ell-2n_1-\cdots-2n_k}} & \le 0 &\mbox{and} \\
\frac{\alpha_C(4)}{2^{\ell-4}}+\sum_{\substack{1\le n_1\le\cdots\le n_k,\;2\le n_k\\ 5<k+2n_1+\cdots+2n_k\le\ell}}\frac{\max\{0,\alpha_C(2n_1,\ldots,2n_k)\}}{12^{k-1}\cdot\pi^{2n_1+\cdots+2n_k-2k-2}\cdot 2^{\ell-2n_1-\cdots-2n_k}} & \le 0.
\end{align*}
\end{theorem}

\begin{proof}
Fix an orientation $C$ of a cycle of even length $\ell$ that satisfies the assumption of the theorem.
As in the proof of Theorems~\ref{thm:QRgeneral} and~\ref{thm:QR22general},
we present an argument for the case $\gamma_C=1$ as the other case is completely symmetric.
Consider a tournamenton $W$ and let $U=W-1/2$.
We again use \eqref{eq:even} and Lemmas~\ref{lm:positive} and~\ref{lm:Q2}
but we bound the terms where all $n_1,\ldots,n_k$ are equal to one and the remaining terms separately.
We bound the terms with $n_1=\cdots=n_k=1$ using $t(Q_2,U)\le 1/12$ as follows:
\begin{align*}
& \sum_{\substack{n_1=\cdots=n_k=1\\ k+2n_1+\cdots+2n_k\le\ell}}\frac{\alpha_C(2n_1,\ldots,2n_k)t(Q_2,U)^k}{2^{\ell-2k}} \\
& \ge \frac{\alpha_C(2)t(Q_2,U)}{2^{\ell-2}}+\sum_{\substack{n_1=\cdots=n_k=1\\ 3<k+2n_1+\cdots+2n_k\le\ell}}\frac{\min\{0,\alpha_C(2n_1,\ldots,2n_k)\}t(Q_2,U)^k}{2^{\ell-2k}} \\
& \ge \frac{\alpha_C(2)t(Q_2,U)}{2^{\ell-2}}+\sum_{\substack{n_1=\cdots=n_k=1\\ 3<k+2n_1+\cdots+2n_k\le\ell}}\frac{\min\{0,\alpha_C(2n_1,\ldots,2n_k)\}t(Q_2,U)}{12^{k-1}\cdot 2^{\ell-2k}} \ge 0.
\end{align*}
The remaining terms are bounded using Lemmas~\ref{lm:positive} and~\ref{lm:Q2} as follows:
\begin{align*}
& \sum_{\substack{1\le n_1\le\cdots\le n_k,\;2\le n_k\\ k+2n_1+\cdots+2n_k\le\ell}} \frac{\alpha_C(2n_1,\ldots,2n_k)}{2^{\ell-2n_1-\cdots-2n_k}}\prod_{i=1}^k t(Q_{2n_i},U)\\
& \ge \frac{\alpha_C(4)t(Q_4,U)}{2^{\ell-4}}+\sum_{\substack{1\le n_1\le\cdots\le n_k,\;2\le n_k\\ 5<k+2n_1+\cdots+2n_k\le\ell}}\frac{\min\{0,\alpha_C(2n_1,\ldots,2n_k)\}}{2^{\ell-2n_1-\cdots-2n_k}}\prod_{i=1}^k t(Q_{2n_i},U)\\
& \ge \frac{\alpha_C(4)t(Q_4,U)}{2^{\ell-4}}+\sum_{\substack{1\le n_1\le\cdots\le n_k,\;2\le n_k\\ 5<k+2n_1+\cdots+2n_k\le\ell}}\frac{\min\{0,\alpha_C(2n_1,\ldots,2n_k)\}}{12^{k-1}\cdot\pi^{2n_1+\cdots+2n_k-2k-2}\cdot 2^{\ell-2n_1-\cdots-2n_k}}t(Q_4,U)\ge 0.
\end{align*}
The above two estimates yield that the value of the sum in \eqref{eq:even} is always non-negative and
so we obtain that
\[t\left(C,\frac{1}{2}+U\right)\ge\frac{1}{2^{\ell}}+t(D_{\ell},U).\]
Hence, $t(C,W)=1/2^{\ell}$ can hold only if $t(D_\ell,U)=0$.
Since $t(D_\ell,U)=0$ holds if and only if $W\equiv 1/2$ (or equivalently $U\equiv 0$) by Lemma~\ref{lm:cycle},
we obtain that the cycle $C$ is quasirandom-forcing.
\end{proof}

We now present our last sufficient condition on an orientation of a cycle to be quasi\-random-forcing.

\begin{theorem}
\label{thm:QRgeneral4}
Let $C$ be an orientation of a cycle of even length $\ell$.
If $\gamma_C=1$, $\alpha_C(2)=0$ and 
\[\frac{\alpha_C(4)}{2^{\ell-4}}+
  \sum_{3\le n_1\le\ell/2-1}\frac{\min\{0,\alpha_C(2n_1)\}}{\pi^{2n_1-4}\cdot 2^{\ell-2n_1}}+
  \hspace{-0.5em}
  \sum_{\substack{1\le n_1\le\cdots\le n_k\\ k+2n_1+\cdots+2n_k\le\ell\\ k\ge 2}}
  \hspace{-0.5em}
  \frac{\min\{0,\alpha_C(2n_1,\ldots,2n_k)\}}{12^{k-2}\cdot\pi^{2n_1+\cdots+2n_k-2k}\cdot 2^{\ell-2n_1-\cdots-2n_k}}\ge 0,\]
then the cycle $C$ is quasirandom-forcing.

Similarly, if $\gamma_C=-1$, $\alpha_C(2)=0$ and
\[\frac{\alpha_C(4)}{2^{\ell-4}}+
  \sum_{3\le n_1\le\ell/2-1}\frac{\max\{0,\alpha_C(2n_1)\}}{\pi^{2n_1-4}\cdot 2^{\ell-2n_1}}+
  \hspace{-0.5em}
  \sum_{\substack{1\le n_1\le\cdots\le n_k\\ k+2n_1+\cdots+2n_k\le\ell\\ k\ge 2}}
  \hspace{-0.5em}
  \frac{\max\{0,\alpha_C(2n_1,\ldots,2n_k)\}}{12^{k-2}\cdot\pi^{2n_1+\cdots+2n_k-2k}\cdot 2^{\ell-2n_1-\cdots-2n_k}}\le 0,\]
then the cycle $C$ is quasirandom-forcing.
\end{theorem}

\begin{proof}
Fix an orientation $C$ of a cycle of even length $\ell$ that satisfies the assumption of the theorem.
We present the proof when $\gamma_C=1$ as an argument in the other case is completely symmetric.
Consider a tournamenton $W$ and let $U=W-1/2$.
We start by showing that
\begin{equation}
t(Q_2,U)^2\le t(Q_4,U).\label{eq:Q22Q4}
\end{equation}
Observe the following identities:
\begin{align*}
t(Q_2,U) & =  \int_{[0,1]^3} U(z,y)U(z,x) \dd z\dd x\dd y = \int_{[0,1]^2} -U^2(x,y) \dd x\dd y\ \mbox{ and}\\
t(Q_4,U) & = \int_{[0,1]^3} U^2(x,y)U^2(x,y') \dd x\dd y\dd y'.
\end{align*}
We obtain using the Cauchy–Schwarz inequality that
\[ \int_{[0,1]^2} -U^2(x,y) \dd y\dd x\le
   \left(\int_{[0,1]} \left(\int_{[0,1]}-U^2(x,y)\dd y\right)^2\dd x\right)^{1/2}.\]
Since it holds for every $x\in [0,1]$ that   
\[\left(\int_{[0,1]} -U^2(x,y)\dd y\right)^2=
  \int_{[0,1]^2} U^2(x,y)U^2(x,y')\dd y\dd y',\]
we derive that
\[ t(Q_2,U)=\int_{[0,1]^2} -U^2(x,y) \dd x\dd y\le
   \left(\int_{[0,1]^3} U^2(x,y)U^2(x,y') \dd x\dd y\dd y'\right)^{1/2}=
   t(Q_4,U)^{1/2},\]
as wanted.
%which implies \eqref{eq:Q22Q4} (recall that $t(P_2,U)=-t(Q_2,U)$).

Lemma~\ref{lm:positive} implies that the following holds for every positive integer $m\ge 2$:
\begin{equation}
t(Q_{2m},U)\le\frac{t(Q_4,U)}{\pi^{2m-4}}.\label{eq:QR1}
\end{equation}
Lemmas~\ref{lm:positive} and~\ref{lm:Q2} together with \eqref{eq:Q22Q4} imply
that the following holds for all positive integers $n_1,\ldots,n_k$, $k\ge 2$:
\begin{equation}
\prod_{i=1}^k t(Q_{2n_i},U)
\le\frac{t(Q_2,U)^2}{12^{k-2}\pi^{2n_1+\cdots+2n_k-2k}}
\le\frac{t(Q_4,U)}{12^{k-2}\pi^{2n_1+\cdots+2n_k-2k}}.\label{eq:QR2}
\end{equation}
We next use \eqref{eq:even}, \eqref{eq:QR1} and \eqref{eq:QR2} and
the assumption of the theorem
to derive the following series of inequalities:
\begin{align*}
&t\left(C,\frac{1}{2}+U\right)=\frac{1}{2^{\ell}}+
                              \sum_{\substack{1\le n_1\le\cdots\le n_k\\ k+2n_1+\cdots+2n_k\le\ell}}\frac{\alpha_C(2n_1,\ldots,2n_k)}{2^{\ell-2n_1-\cdots-2n_k}}\prod_{i=1}^k t(Q_{2n_i},U)+t(D_{\ell},U)\\
			     &\ge\frac{1}{2^{\ell}}+\frac{\alpha_C(4)t(Q_4,U)}{2^{\ell-4}}+t(D_{\ell},U)
                             +\!\!\!\sum_{\substack{1\le n_1\le\cdots\le n_k\\ 5<k+2n_1+\cdots+2n_k\le\ell}}\!\!\frac{\min\{0,\alpha_C(2n_1,\ldots,2n_k)\}}{2^{\ell-2n_1-\cdots-2n_k}}\prod_{i=1}^k t(Q_{2n_i},U)\\
			     &\ge\frac{1}{2^{\ell}}+\frac{\alpha_C(4)t(Q_4,U)}{2^{\ell-4}}+t(D_{\ell},U)
                             +\sum_{3\le n_1\le\ell/2-1}\frac{\min\{0,\alpha_C(2n_1)\}t(Q_4,U)}{\pi^{2n_1-4}\cdot 2^{\ell-2n_1}}\\
			     &\qquad+\sum_{\substack{1\le n_1\le\cdots\le n_k\\ k+2n_1+\cdots+2n_k\le\ell\\ k\ge 2}}\frac{\min\{0,\alpha_C(2n_1,\ldots,2n_k)\}t(Q_4,U)}{12^{k-2}\cdot\pi^{2n_1+\cdots+2n_k-2k}\cdot 2^{\ell-2n_1-\cdots-2n_k}}\\
			     &\ge\frac{1}{2^{\ell}}+t(D_{\ell},U).
\end{align*}
Hence, $t(C,W)=1/2^{\ell}$ can hold only if $t(D_\ell,U)=0$.
Since $t(D_\ell,U)=0$ holds if and only if $W\equiv 1/2$ (or equivalently $U\equiv 0$) by Lemma~\ref{lm:cycle},
we obtain that the cycle $C$ is quasirandom-forcing.
\end{proof}

\section{Negative results}
\label{sec:negative}

In this section,
we will present conditions that imply that an orientation of an even cycle is not quasirandom-forcing.
We start with showing that no orientation of an odd cycle is quasirandom-forcing.

\begin{theorem}
\label{thm:odd}
No orientation of an odd cycle is quasirandom-forcing.
\end{theorem}

\begin{proof}
Fix an orientation $C$ of an odd cycle.
Let $U_3$ be the antisymmetric kernel (also see Figure~\ref{fig:U23}) defined as
\[U_3(x,y)=\begin{cases}
           1/2 & \mbox{if $(x,y)\in (0,1/3)\times (2/3,1)$,}\\
           1/2 & \mbox{if $(x,y)\in (1/3,2/3)\times (0,1/3)$,}\\
           1/2 & \mbox{if $(x,y)\in (2/3,1)\times (1/3,2/3)$,}\\
           -1/2 & \mbox{if $(x,y)\in (0,1/3)\times (1/3,2/3)$,}\\
           -1/2 & \mbox{if $(x,y)\in (1/3,2/3)\times (2/3,1)$,}\\
           -1/2 & \mbox{if $(x,y)\in (2/3,1)\times (0,1/3)$, and}\\
	   0 & \mbox{otherwise.}
           \end{cases}\]
Observe that the following holds for every $x\in [0,1]$:
\[\int_{[0,1]}U_3(x,y)\dd y=0.\]
Let $P_{m}$ be a directed path with $m$ edges. Note that
\begin{align*}
t(P_m,U_3) &=\int_{[0,1]^{m+1}}\prod_{i=1}^m U_3(x_{i-1},x_i)\dd x_0\dots\dd x_m\\
           &\hspace{-8.5pt}=\int_{[0,1]}\left(\int_{[0,1]^{m-1}}\prod_{i=1}^{m-1} U_3(x_{i-1},x_i)\dd x_0\dots\dd x_{m-2}\right)\left(\int_{[0,1]}U_3(x_{m-1},x_m)\dd x_m\right)\dd x_{m-1}\\
           &\hspace{-8.5pt}=\int_{[0,1]}\left(\int_{[0,1]^{m-1}}\prod_{i=1}^{m-1} U_3(x_{i-1},x_i)\dd x_0\dots\dd x_{m-2}\right)\cdot 0\dd x_{m-1}=0,
\end{align*}
which yields by Proposition~\ref{prop:reorient} that $t(Q_{2k},U_3)=0$ for every $k\in\NN$.
Hence, we obtain using~\eqref{eq:odd} that $t(C,1/2+U_3)=2^{-\|C\|} = t(C,1/2)$ for the orientation $C$ and
so $C$ is not quasirandom-forcing.
\end{proof}

\begin{figure}[t]
\begin{center}
\epsfbox{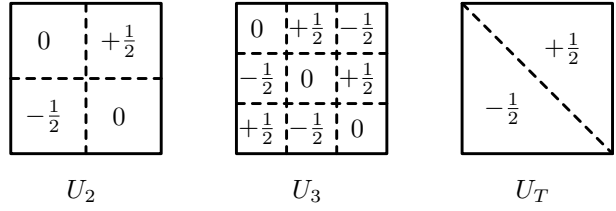}
\end{center}
\caption{The kernels $U_2$, $U_3$ and $U_T$. The origin is in the top left corner.}
\label{fig:U23}
\end{figure}

A standard way of showing that a particular (oriented) graph $H$ is not quasirandom-forcing
is to identify a construction with the lower density of $H$ and
a construction with the larger density of $H$
compared to the expected density in the random construction, and
then interpolate between the two constructions to show that $H$ is not quasirandom-forcing.
We summarize this approach in the next lemma.

\begin{lemma}
\label{lm:W12}
Let $C$ be an orientation of a cycle.
If there exist tournamentons $W_1$ and $W_2$ such that
$t(C,W_1)<2^{-\|C\|}$ and $t(C,W_2)>2^{-\|C\|}$,
then $C$ is not quasirandom-forcing.
\end{lemma}

\begin{proof}
Fix an orientation $C$ of a cycle of length $\ell$ and
the tournamentons $W_1$ and $W_2$.
Since $t(C,W_1)\not=2^{-\ell}$ and $t(C,W_2)\not=2^{-\ell}$,
neither $W_1$ nor $W_2$ is equal to $1/2$ almost everywhere.

We next define tournamentons $W_\alpha$ for $\alpha\in (1,2)$ that
continuously transform $W_1$ into $W_2$.
For $\alpha\in (1,2)$, let $W_\alpha$ be the tournamenton defined as
\[W_\alpha(x,y)=\begin{cases}
                W_2\left(\frac{x}{\alpha-1},\frac{y}{\alpha-1}\right) & \mbox{if $(x,y)\in (0,\alpha-1)^2$,}\\
	        W_1\left(\alpha-1+\frac{x-(\alpha-1)}{2-\alpha},\alpha-1+\frac{y-(\alpha-1)}{2-\alpha}\right) & \mbox{if $(x,y)\in (\alpha-1,1)^2$, and}\\
                1/2 & \mbox{otherwise.}
	        \end{cases}\]
Note that none of the tournamentons $W_\alpha$, $\alpha\in [1,2]$, satisfies that $W_\alpha\equiv 1/2$.
Next define $g:[1,2]\to [0,1]$ as $g(\alpha)=t(C,W_\alpha)$.
Since $g$ is a continuous function on $[1,2]$, $g(1)<2^{-\ell}$ and $g(2)>2^{-\ell}$,
there exists $\alpha_0\in (1,2)$ such that $g(\alpha_0)=2^{-\ell}$.
Since $t(C,W_{\alpha_0})=2^{-\ell}$ but $W_{\alpha_0}$ is not equal to $1/2$ almost everywhere,
the orientation $C$ is not quasirandom-forcing.
\end{proof}

The next lemma guarantees the existence of one of two tournamentons needed to apply Lemma~\ref{lm:W12}.

\begin{lemma}
\label{lm:W1}
Let $C$ be an orientation of an even cycle.
If $\gamma_C=1$, then there exists a tournamenton $W$ such that $t(C,W)>2^{-\|C\|}$, and
if $\gamma_C=-1$, then there exists a tournamenton $W$ such that $t(C,W)<2^{-\|C\|}$.
\end{lemma}

\begin{proof}
Fix an orientation $C$ of an even cycle and let $\ell$ be the length of $C$.
Let $U_3$ be the antisymmetric kernel defined in the proof of Theorem~\ref{thm:odd} and visualized in Figure~\ref{fig:U23}.
As argued in the proof of Theorem~\ref{thm:odd},
it holds that $t(Q_{2k},U_3)=0$ for every $k\in\NN$.
Hence, the identity \eqref{eq:even} yields that
\[t\left(C,\frac{1}{2}+U_3\right)=\frac{1}{2^{\ell}}+\gamma_C t(D_{\ell},U_3).\]
Since $t(D_{\ell},U_3)>0$ by Lemma~\ref{lm:cycle},
it follows that the tournamenton $W=1/2+U_3$ has the property claimed in the statement of the lemma.
\end{proof}

Using Lemmas~\ref{lm:W12} and~\ref{lm:W1}, we obtain our first condition 
implying that an orientation of an even cycle is not quasirandom-forcing.
To state the next theorem, we first define the antisymmetric kernel $U_T$ as
\[U_T(x,y)=\begin{cases}
           1/2 & \mbox{if $0\le x<y\le 1$,}\\
	   -1/2 & \mbox{if $0\le y<x\le 1$, and}\\
	   0 & \mbox{otherwise.}
	   \end{cases}\]
The kernel $U_T$ is visualized in Figure~\ref{fig:U23}.
Note that the tournamenton $1/2+U_T$ is a limit of transitive tournaments.

\begin{theorem}
\label{thm:transitive}
Let $C$ be an orientation of an even cycle of length $\ell$.
If $t(C,1/2+U_T)\le 2^{-\ell}$ and $\gamma_C=1$, then $C$ is not quasirandom-forcing.
Similarly, 
if $t(C,1/2+U_T)\ge 2^{-\ell}$ and $\gamma_C=-1$, then $C$ is not quasirandom-forcing.
\end{theorem}

\begin{proof}
Fix an orientation $C$ of a cycle with even length $\ell$.
We will analyze the case $\gamma_C=1$ only as
the argument concerning the case $\gamma_C=-1$ is completely symmetric.
Hence, we assume that $\gamma_C=1$ and $t(C,1/2+U_T)\le 2^{-\ell}$ in the rest of the proof.
If $t(C,1/2+U_T)=2^{-\ell}$,
then the tournamenton $W=1/2+U_T$ is a tournamenton such that $t(C,W)=2^{-\ell}$ but $W$ is not equal to $1/2$ almost everywhere;
hence, $C$ is not quasirandom-forcing.
If $t(C,1/2+U_T)<2^{-\ell}$, we set $W_1=1/2+U_T$.
Since $\gamma_C=1$,
Lemma~\ref{lm:W1} implies that there exists a tournamenton $W_2$ such that $t(C,W_2)>2^{-\ell}$.
The existence of tournamentons $W_1$ and $W_2$ such that $t(C,W_1)<2^{-\ell}$ and $t(C,W_2)>2^{-\ell}$
implies that $C$ is not quasirandom-forcing by Lemma~\ref{lm:W12},
which completes the proof of the theorem.
\end{proof}

We next present our second condition implying that an orientation of an even cycle is not quasirandom-forcing.

\begin{theorem}
\label{thm:neg}
Let $C$ be an orientation of an even cycle.
If both $\alpha_C(2)$ and $\gamma_C$ are non-zero and have different signs,
then $C$ is not quasirandom-forcing.
\end{theorem}

\begin{proof}
Fix an orientation $C$ of an even cycle and let $\ell$ be the length of $C$.
We will analyze the case $\alpha_C(2)>0$ and $\gamma_C=-1$;
the argument concerning the other case $\alpha_C(2)<0$ and $\gamma_C=1$ is completely symmetric.

Let $U_2$ be the antisymmetric kernel (also see Figure~\ref{fig:U23}) defined as
\[U_2(x,y)=\begin{cases}
	   1/2 & \mbox{if $(x,y)\in (1/2,1)\times (0,1/2)$,}\\
           -1/2 & \mbox{if $(x,y)\in (0,1/2)\times (1/2,1)$, and}\\
	   0 & \mbox{otherwise.}
	   \end{cases}\]
Observe that $t(Q_2,U_2)=1/16$.
The identity \eqref{eq:even} implies that
\[t(C,1/2+\varepsilon U_2)=\frac{1}{2^{\ell}}+\frac{\alpha_C(2)\varepsilon^2}{2^{\ell-2}\cdot 16}+O(\varepsilon^4).\]
Hence, there exists $\varepsilon_0\in (0,1)$ such that $t(C,1/2+\varepsilon_0 U_2)>2^{-\ell}$;
set $W_2$ to be the tournamenton $1/2+\varepsilon_0 U_2$.
Since $\gamma_C=-1$,
Lemma~\ref{lm:W1} implies that there exists a tournamenton $W_1$ such that $t(C,W_1)<2^{-\ell}$.
The existence of tournamentons $W_1$ and $W_2$ such that $t(C,W_1)<2^{-\ell}$ and $t(C,W_2)>2^{-\ell}$
implies that $C$ is not quasirandom-forcing by Lemma~\ref{lm:W12}.
\end{proof}

The third condition implying that an orientation of an even cycle is not quasirandom-forcing
is given in the next theorem.

\begin{theorem}
\label{thm:neg4}
Let $C$ be an orientation of an even cycle of length at least 6.
If $\alpha_C(2)=0$, and both $\alpha_C(4)+\alpha_C(2,2)$ and $\gamma_C$ are non-zero and have different signs,
then $C$ is not quasirandom-forcing.
\end{theorem}

\begin{proof}
Fix an orientation $C$ of an even cycle and let $\ell$ be the length of the cycle.
The proof is similar to the proof of Theorem~\ref{thm:neg}.
We will again analyze the case $\alpha_C(2)>0$ and $\gamma_C=-1$ only as
the argument concerning the other case $\alpha_C(2)<0$ and $\gamma_C=1$ is completely symmetric.

Since $\gamma_C$ is equal to $-1$,
Lemma~\ref{lm:W1} implies that there exists a tournamenton $W_1$ such that $t(C,W_1)<2^{-\ell}$.
Let $U_2$ be the antisymmetric kernel defined in the proof of Theorem~\ref{thm:neg} and visualized in Figure~\ref{fig:U23}.
Observe that $t(Q_2,U_2)=1/16$ and $t(Q_4,U_2)=1/256$.
The identity \eqref{eq:even} implies (note that $\alpha_C(2)=0$) that
\[t(C,1/2+\varepsilon U_2)=\frac{1}{2^{\ell}}+\frac{\left(\alpha_C(4)+\alpha_C(2,2)\right)\varepsilon^4}{2^{\ell-4}\cdot 256}+O(\varepsilon^6).\]
Hence, there exists $\varepsilon_0\in (0,1)$ such that $t(C,1/2+\varepsilon_0 U_2)>2^{-\ell}$. Set $W_2$ to be the tournamenton $1/2+\varepsilon_0 U_2$.
The existence of tournamentons $W_1$ and $W_2$ such that $t(C,W_1)<2^{-\ell}$ and $t(C,W_2)>2^{-\ell}$
implies that $C$ is not quasirandom-forcing by Lemma~\ref{lm:W12}.
\end{proof}

In Theorems~\ref{thm:neg} and~\ref{thm:neg4},
we applied Lemma~\ref{lm:W12} with a tournamenton whose existence is guaranteed by Lemma~\ref{lm:W1} and
a tournamenton whose existence is derived by choosing a suitable antisymmetric kernel $U$ and analyzing the identity \eqref{eq:even}.
The final theorem of this section extracts the idea of choosing a suitable antisymmetric kernel $U$
to give a condition implying that an orientation of an even cycle is not quasirandom-forcing.

\begin{theorem}
\label{thm:negsum}
Let $C$ be an orientation of an even cycle of length $\ell$.
If $\gamma_C=1$ and there exists an antisymmetric kernel $U:[0,1]^2\to [-1/2,1/2]$ such that
\[ \sum_{\substack{1\le n_1\le\cdots\le n_k\\ k+2n_1+\cdots+2n_k\le\ell}}
                   \frac{\alpha_C(2n_1,\ldots,2n_k)}{2^{\ell-2n_1-\cdots-2n_k}}\prod_{i=1}^k t(Q_{2n_i},U)+
   \gamma_C\cdot t(D_{\ell},U)<0,\]
then $C$ is not quasirandom-forcing.

Similarly,
if $\gamma_C=-1$ and there exists an antisymmetric kernel $U:[0,1]^2\to [-1/2,1/2]$ such that
\[ \sum_{\substack{1\le n_1\le\cdots\le n_k\\ k+2n_1+\cdots+2n_k\le\ell}}
                   \frac{\alpha_C(2n_1,\ldots,2n_k)}{2^{\ell-2n_1-\cdots-2n_k}}\prod_{i=1}^k t(Q_{2n_i},U)+
   \gamma_C\cdot t(D_{\ell},U)>0,\]
then $C$ is not quasirandom-forcing.
\end{theorem}

\begin{proof}
Fix an orientation $C$ of an even cycle and let $\ell$ be the length of $C$.
We will analyze the case $\gamma_C=1$;
the argument concerning the other case $\gamma_C=-1$ is completely symmetric.

Let $U$ be the kernel with the properties given in the case $\gamma_C=1$ of the statement of the theorem.
The identity \eqref{eq:even} implies that $t(C,1/2+U)<2^{-\ell}$.
We set $W_1$ to be the tournamenton $1/2+U$.
Since $\gamma_C$ is equal to $1$,
Lemma~\ref{lm:W1} implies that there exists a tournamenton $W_2$ such that $t(C,W_2)>2^{-\ell}$.
The existence of tournamentons $W_1$ and $W_2$ such that $t(C,W_1)<2^{-\ell}$ and $t(C,W_2)>2^{-\ell}$
implies that $C$ is not quasirandom-forcing by Lemma~\ref{lm:W12}.
\end{proof}

Two particular kernels that we will apply Theorem~\ref{thm:negsum} with are the kernels $U_2$ and $U_T$,
which are visualized in Figure~\ref{fig:U23}.
The kernel $U_2$ is the antisymmetric kernel defined in the proof of Theorem~\ref{thm:neg}, and
the kernel $U_T$ is the antisymmetric kernel defined before Theorem~\ref{thm:transitive}.
For reference, we state the densities of oriented graphs $Q_{2n}$ and $D_{2n}$ in the kernel $U_2$:
\begin{align*}
  t(Q_{2n},U_2) & =2^{-4n} \mbox{ for every integer $n\ge 1$, and} \\
  t(D_{2n},U_2) & =2^{-4n} \mbox{ for every integer $n\ge 2$.}
\end{align*}
Similarly, we state the densities of oriented graphs $Q_{2n}$ for small values of $n$ in the kernel~$U_T$:
%\begin{align*}
%  t(Q_2,U_T) & =\frac{2}{2^2\cdot 3!}=\frac{1}{12}\mbox{,}\\
%  t(Q_4,U_T) & =\frac{16}{2^4\cdot 5!}=\frac{1}{120}\mbox{,}\\
%  t(Q_6,U_T) & =\frac{272}{2^6\cdot 7!}=\frac{17}{4\cdot 7!}\mbox{,}\\
%  t(Q_8,U_T) & =\frac{7936}{2^8\cdot 9!}=\frac{31}{9!}\mbox{,}\\
%  t(Q_{10},U_T) & =\frac{353792}{2^{10}\cdot 11!}=\frac{691}{2\cdot 11!}\mbox{, and}\\
%  t(D_{2n},U_T) & = \frac{1}{4}t(Q_{2n-2},U_T) \mbox{ for every integer $n\ge 2$.}
%\end{align*}
\[t(Q_2,U_T)=\frac{2}{2^2\cdot 3!}=\frac{1}{12}\mbox{, }
  t(Q_4,U_T)=\frac{16}{2^4\cdot 5!}=\frac{1}{120}\mbox{, }
  t(Q_6,U_T)=\frac{272}{2^6\cdot 7!}=\frac{17}{4\cdot 7!}\mbox{, }\]
\[  t(Q_8,U_T)=\frac{7936}{2^8\cdot 9!}=\frac{31}{9!}\mbox{, }
  t(Q_{10},U_T)=\frac{353792}{2^{10}\cdot 11!}=\frac{691}{2\cdot 11!}\mbox{, and}\]
\[t(D_{2n},U_T) = \frac{1}{4}t(Q_{2n-2},U_T) \mbox{ for every integer $n\ge 2$.}\]
The values of $t(Q_{2n},U_T)$ can be determined by looking at all bijections from $V(Q_{2n})$ to $\{1, 2, \ldots, 2n+1\}$ and checking the parity of the number of edges keeping the order.

\section{Classification of orientations}
\label{sec:specific}

The general results presented in Sections~\ref{sec:general} and~\ref{sec:negative}
are strong enough to provide a full classification which orientations of cycles of length four, six and eight
are quasirandom-forcing.
We summarize results concerning orientations of cycles of length four, six and eight in Tables~\ref{tab:4}--\ref{tab:8}.
In particular, $3$ out of $4$ non-isomorphic orientations of the cycle of length four are quasirandom-forcing,
$4$ out of $8$ non-isomorphic orientations of the cycle of length six are quasirandom-forcing, and
$9$ out of $18$ non-isomorphic orientations of the cycle of length eight are quasirandom-forcing.
Tables~\ref{tab:4}--\ref{tab:8} also contain
the values of the coefficients $\alpha_C(2n_1,\ldots,2n_k)$ and $\gamma_C$ for each orientation $C$
so that it is easy to verify that the assumptions of the used theorems are satisfied.

\begin{table}
\begin{center}
\begin{tabular}{|c|c|c|cl|}
\hline
$C$ & $\alpha_C(2)$ & $\gamma_C$ & \multicolumn{2}{l|}{quasirandom-forcing} \\
\hline
\cepsfbox{qcycle-101.mps} & $-4$ & $+1$ & no  & by Theorem~\ref{thm:neg} \\
% 0/24
\cepsfbox{qcycle-102.mps} &  $0$ & $-1$ & yes & by Theorem~\ref{thm:QRgeneral} \\
% 1/24
\cepsfbox{qcycle-103.mps} &  $0$ & $+1$ & yes & by Theorem~\ref{thm:QRgeneral} \\
% 2/24
\cepsfbox{qcycle-104.mps} & $+4$ & $+1$ & yes & by Theorem~\ref{thm:QRgeneral} \\[4pt]
% 4/24
\hline
\end{tabular}
\end{center}
\caption{Classification of all orientations of a cycle of length four.}
\label{tab:4}
\end{table}

\begin{table}
\begin{center}
\begin{tabular}{|c|ccc|c|cl|}
\hline
$C$ & $(2)$ & $(4)$ & $(2,2)$ & $\gamma_C$ & \multicolumn{2}{l|}{quasirandom-forcing} \\
\hline
\cepsfbox{qcycle-201.mps} & $-6$ & $+6$ & $+3$ & $-1$ & yes & by Theorem~\ref{thm:QRgeneral} ($-0.161$) \\
% 0/720
\cepsfbox{qcycle-202.mps} & $-2$ & $-2$ & $-1$ & $+1$ & no  & by Theorem~\ref{thm:neg} \\
% 1/720
\cepsfbox{qcycle-203.mps} & $-2$ & $+2$ & $-1$ & $-1$ & yes & by Theorem~\ref{thm:QRgeneral} ($-0.074$) \\
% 4/720
\cepsfbox{qcycle-204.mps} & $+2$ & $-2$ & $-1$ & $-1$ & no  & by Theorem~\ref{thm:neg} \\
% 13/720
\cepsfbox{qcycle-205.mps} & $-2$ & $-2$ & $+3$ & $+1$ & no  & by Theorem~\ref{thm:neg} \\
% 6/720
\cepsfbox{qcycle-206.mps} & $+2$ & $-2$ & $+3$ & $-1$ & no  & by Theorem~\ref{thm:neg} \\
% 18/720
\cepsfbox{qcycle-207.mps} & $+2$ & $+2$ & $-1$ & $+1$ & yes & by Theorem~\ref{thm:QRgeneral} ($+0.104$) \\
% 22/720
\cepsfbox{qcycle-208.mps} & $+6$ & $+6$ & $+3$ & $+1$ & yes & by Theorem~\ref{thm:QRgeneral} ($+0.375$) \\[4pt]
% 48/720
\hline
\end{tabular}
\end{center}
\caption{Classification of all orientations of a cycle of length six.
         In case the orientation is quasirandom-forcing by Theorem~\ref{thm:QRgeneral},
	 the approximate value of the sum from the statement of Theorem~\ref{thm:QRgeneral} is given in the parenthesis.
	 For brevity, the coefficients $\alpha_C(2n_1,\ldots,2n_k)$ are represented just by $(2n_1,\ldots,2n_k)$ in the first line.}
\label{tab:6}
\end{table}

\begin{table}
\begin{center}
\begin{tabular}{|c|ccccc|c|cl|}
\hline
$C$ & $(2)$ & $(4)$ & $(2,2)$ & $(6)$ & $(2,4)$ & $\gamma_C$ & \multicolumn{2}{l|}{quasirandom-forcing} \\
\hline
  \cepsfbox{qcycle-301.mps} & $-8$ & $+8$ & $+12$ & $-8$ & $-8$ & $+1$ & no  & by Theorem~\ref{thm:neg} \\
  % 0/40320
  \cepsfbox{qcycle-302.mps} & $-4$ & $+0$ &  $+0$ & $+4$ & $+4$ & $-1$ & yes & by Theorem~\ref{thm:QRgeneral} ($-0.044$) \\
  % 1/40320
  \cepsfbox{qcycle-303.mps} & $-4$ & $+4$ &  $+0$ & $-4$ & $+0$ & $+1$ & no  & by Theorem~\ref{thm:neg} \\
  % 6/40320
  \cepsfbox{qcycle-304.mps} & $+0$ & $+0$ &  $-8$ & $+0$ & $+0$ & $+1$ & no  & by Theorem~\ref{thm:neg4} \\
  % 26/40320
  \cepsfbox{qcycle-305.mps} & $-4$ & $+0$ &  $+4$ & $+4$ & $-4$ & $-1$ & yes & by Theorem~\ref{thm:QRgeneral} ($-0.031$) \\
  % 15/40320
  \cepsfbox{qcycle-306.mps} & $+0$ & $-4$ &  $+0$ & $+0$ & $-4$ & $+1$ & no  & by Theorem~\ref{thm:neg4} \\
  % 54/40320
  \cepsfbox{qcycle-307.mps} & $+0$ & $+0$ &  $-4$ & $+0$ & $+0$ & $-1$ & yes & by Theorem~\ref{thm:QRsigns} \\
  % 79/40320
  \cepsfbox{qcycle-308.mps} & $-4$ & $+0$ &  $+4$ & $-4$ & $+4$ & $+1$ & no  & by Theorem~\ref{thm:neg} \\
  % 20/40320
  \cepsfbox{qcycle-309.mps} & $+0$ & $-8$ &  $+4$ & $+0$ & $+0$ & $+1$ & no  & by Theorem~\ref{thm:neg4} \\
  % 68/40320
  \cepsfbox{qcycle-310.mps} & $+0$ & $+0$ &  $+0$ & $+0$ & $+0$ & $-1$ & yes & by Theorem~\ref{thm:QRgeneral} ($0$) \\
  % 149/40320
  \cepsfbox{qcycle-311.mps} & $+4$ & $+0$ &  $+0$ & $-4$ & $-4$ & $-1$ & no  & by Theorem~\ref{thm:neg} \\
  % 297/40320
  \cepsfbox{qcycle-312.mps} & $+0$ & $-4$ &  $+0$ & $+0$ & $+4$ & $+1$ & no  & by Theorem~\ref{thm:neg4} \\
  % 110/40320
  \cepsfbox{qcycle-313.mps} & $+0$ & $+0$ &  $+0$ & $+0$ & $+0$ & $+1$ & yes & by Theorem~\ref{thm:QRgeneral} ($0$) \\
  % 166/40320
  \cepsfbox{qcycle-314.mps} & $+4$ & $+0$ &  $+4$ & $-4$ & $+4$ & $-1$ & no  & by Theorem~\ref{thm:neg} \\
  % 423/40320
  \cepsfbox{qcycle-315.mps} & $+4$ & $+4$ &  $+0$ & $+4$ & $+0$ & $+1$ & yes & by Theorem~\ref{thm:QRgeneral} ($+0.063$) \\
  % 494/40320
  \cepsfbox{qcycle-316.mps} & $+4$ & $+0$ &  $+4$ & $+4$ & $-4$ & $+1$ & yes & by Theorem~\ref{thm:QRgeneral} ($+0.054$) \\
  % 452/40320
  \cepsfbox{qcycle-317.mps} & $+0$ & $+8$ &  $-4$ & $+0$ & $+0$ & $+1$ & yes & by Theorem~\ref{thm:QRgeneral4} ($+0.250$) \\
  % 264/40320
  \cepsfbox{qcycle-318.mps} & $+8$ & $+8$ & $+12$ & $+8$ & $+8$ & $+1$ & yes & by Theorem~\ref{thm:QRgeneral} ($+0.125$) \\[4pt]
  % 1088/40320
\hline
\end{tabular}
\end{center}
\caption{Classification of all orientations of a cycle of length eight.
         In case the orientation is quasirandom-forcing by Theorem~\ref{thm:QRgeneral} or Theorem~\ref{thm:QRgeneral4},
	 the approximate value of the sum from the statement of the respective theorem is given in the parenthesis.
	 For brevity, the coefficients $\alpha_C(2n_1,\ldots,2n_k)$ are represented just by $(2n_1,\ldots,2n_k)$ in the first line.}
\label{tab:8}
\end{table}

The rest of the section is devoted to the classification of orientations of the cycle of length ten.
There are $44$ non-isomorphic orientations of the cycle of length ten.
There are $15$ orientations that are quasirandom-forcing by Theorem~\ref{thm:QRgeneral};
these orientations are listed in Table~\ref{tab:10:QRgeneral}.
There are $23$ orientations that are not quasirandom-forcing:
$18$ orientations are not quasirandom-forcing by Theorem~\ref{thm:neg} and $5$ by Theorem~\ref{thm:negsum}.
These orientations are listed in Table~\ref{tab:10:neg} and in Table~\ref{tab:10:negsum}, respectively.

\begin{table}
\begin{center}
\begin{tabular}{|c|ccccccccc|c|c|}
\hline
$C$ & $(2)$ & $(4)$ & $(2,2)$ & $(6)$ & $(2,4)$ & $(2,2,2)$ & $(8)$ & $(2,6)$ & $(4,4)$ & $\gamma_C$ & sum \\
\hline
  \cepsfbox{qcycle-403.mps} & $-6$ &  $+6$ &  $+5$ &  $-6$ &  $-6$ &  $+2$ &  $+6$ &  $+2$ &  $+1$ &  $-1$ & $-0.00436$ \\
  % 8/3628800
  \cepsfbox{qcycle-406.mps} & $-2$ &  $-2$ &  $-3$ &  $+2$ &  $+6$ &  $-6$ &  $+2$ &  $+6$ &  $+1$ &  $-1$ & $-0.00135$ \\
  % 118/3628800
  \cepsfbox{qcycle-408.mps} & $-6$ &  $+2$ &  $+9$ &  $-2$ &  $-2$ &  $-2$ &  $+6$ &  $-2$ &  $-3$ &  $-1$ & $-0.00699$ \\
  % 56/3628800
  \cepsfbox{qcycle-409.mps} & $-2$ &  $-6$ &  $+1$ &  $+6$ &  $+2$ &  $-2$ &  $+2$ &  $+2$ &  $+1$ &  $-1$ & $-0.00044$ \\
  % 208/3628800
  \cepsfbox{qcycle-412.mps} & $-2$ &  $-2$ &  $-3$ &  $+2$ &  $+6$ &  $+2$ &  $+2$ &  $-2$ &  $-3$ &  $-1$ & $-0.00197$ \\
  % 406/3628800
  \cepsfbox{qcycle-413.mps} & $-2$ &  $+2$ &  $-3$ &  $-2$ &  $+2$ &  $-2$ &  $+2$ &  $+2$ &  $-3$ &  $-1$ & $-0.00264$ \\
  % 631/3628800
  \cepsfbox{qcycle-418.mps} & $-2$ &  $-2$ &  $+1$ &  $+2$ &  $-2$ &  $+2$ &  $+2$ &  $-2$ &  $+1$ &  $-1$ & $-0.00363$ \\
  % 973/3628800
  \cepsfbox{qcycle-432.mps} & $+2$ &  $-2$ &  $+1$ &  $-2$ &  $+2$ &  $-2$ &  $+2$ &  $-2$ &  $+1$ &  $+1$ & $+0.00207$ \\
  % 5193/3628800
  \cepsfbox{qcycle-434.mps} & $-2$ &  $+2$ &  $+1$ &  $-2$ &  $+2$ &  $-2$ &  $+2$ &  $-6$ &  $+5$ &  $-1$ & $-0.00070$ \\
  % 2200/3628800
  \cepsfbox{qcycle-436.mps} & $+2$ &  $+2$ &  $-3$ &  $+2$ &  $-2$ &  $+2$ &  $+2$ &  $+2$ &  $-3$ &  $+1$ & $+0.00221$ \\
  % 6171/3628800
  \cepsfbox{qcycle-439.mps} & $+2$ &  $+2$ &  $+1$ &  $+2$ &  $-2$ &  $+2$ &  $+2$ &  $-6$ &  $+5$ &  $+1$ & $+0.00547$ \\
  % 7740/3628800
  \cepsfbox{qcycle-441.mps} & $+6$ &  $+6$ &  $+5$ &  $+6$ &  $+6$ &  $-2$ &  $+6$ &  $+2$ &  $+1$ &  $+1$ & $+0.02257$ \\
  % 17998/3628800
  \cepsfbox{qcycle-442.mps} & $+6$ &  $+2$ &  $+9$ &  $+2$ &  $+2$ &  $+2$ &  $+6$ &  $-2$ &  $-3$ &  $+1$ & $+0.02237$ \\
  % 16306/3628800
  \cepsfbox{qcycle-443.mps} & $+2$ &  $+6$ &  $-3$ &  $+6$ &  $+2$ &  $-2$ &  $+2$ &  $-2$ &  $+1$ &  $+1$ & $+0.00261$ \\
  % 8928/3628800
  \cepsfbox{qcycle-444.mps} & $+10$ &  $+10$ &  $+25$ &  $+10$ &  $+30$ &  $+10$ &  $+10$ &  $+10$ &  $+5$ &  $+1$ & $+0.03906$ \\[5pt]
  % 39680/3628800
\hline
\end{tabular}
\end{center}
\caption{Orientations of a cycle of length ten that are quasirandom-forcing by Theorem~\ref{thm:QRgeneral};
         the approximate value of the sum from the statement of the theorem is given in the last column.
	 For brevity, the coefficients $\alpha_C(2n_1,\ldots,2n_k)$ are represented just by $(2n_1,\ldots,2n_k)$ in the first line.}
\label{tab:10:QRgeneral}
\end{table}

\begin{table}
\begin{center}
\begin{tabular}{|c|ccccccccc|c|}
\hline
$C$ & $(2)$ & $(4)$ & $(2,2)$ & $(6)$ & $(2,4)$ & $(2,2,2)$ & $(8)$ & $(2,6)$ & $(4,4)$ & $\gamma_C$ \\
\hline
  \cepsfbox{qcycle-402.mps} & $-6$ &  $+2$ &  $+5$ &  $+2$ &  $+6$ &  $+2$ &  $-6$ &  $-6$ &  $-3$ &  $+1$ \\
  % 1/3628800
  \cepsfbox{qcycle-405.mps} & $-6$ &  $+2$ &  $+9$ &  $+2$ &  $-2$ &  $-6$ &  $-6$ &  $+2$ &  $+1$ &  $+1$ \\
  % 28/3628800
  \cepsfbox{qcycle-407.mps} & $-2$ &  $+2$ &  $-7$ &  $+2$ &  $+2$ &  $+2$ &  $-2$ &  $-2$ &  $+1$ &  $+1$ \\
  % 188/3628800
  \cepsfbox{qcycle-410.mps} & $-2$ &  $+2$ &  $-3$ &  $+2$ &  $-6$ &  $+2$ &  $-2$ &  $-2$ &  $+1$ &  $+1$ \\
  % 503/3628800
  \cepsfbox{qcycle-414.mps} & $-6$ &  $+2$ &  $+9$ &  $+2$ &  $-10$ &  $+2$ &  $-6$ &  $+2$ &  $+5$ &  $+1$ \\
  % 70/3628800
  \cepsfbox{qcycle-416.mps} & $-2$ &  $-2$ &  $+1$ &  $-2$ &  $+6$ &  $-2$ &  $-2$ &  $+2$ &  $-3$ &  $+1$ \\
  % 791/3628800
  \cepsfbox{qcycle-419.mps} & $+2$ &  $+2$ &  $-7$ &  $-2$ &  $-2$ &  $-2$ &  $-2$ &  $-2$ &  $+1$ &  $-1$ \\
  % 2838/3628800
  \cepsfbox{qcycle-420.mps} & $+2$ &  $-2$ &  $-3$ &  $+2$ &  $-6$ &  $-6$ &  $-2$ &  $+2$ &  $+1$ &  $-1$ \\
  % 2388/3628800
  \cepsfbox{qcycle-421.mps} & $-2$ &  $-2$ &  $-3$ &  $-2$ &  $+6$ &  $+6$ &  $-2$ &  $+2$ &  $+1$ &  $+1$ \\
  % 518/3628800
  \cepsfbox{qcycle-423.mps} & $-2$ &  $-2$ &  $+1$ &  $-2$ &  $+6$ &  $-2$ &  $-2$ &  $+2$ &  $+1$ &  $+1$ \\
  % 1043/3628800
  \cepsfbox{qcycle-427.mps} & $+2$ &  $-2$ &  $+1$ &  $+2$ &  $-6$ &  $+2$ &  $-2$ &  $+2$ &  $+1$ &  $-1$ \\
  % 5013/3628800
  \cepsfbox{qcycle-428.mps} & $+2$ &  $-2$ &  $+1$ &  $+2$ &  $-6$ &  $+2$ &  $-2$ &  $+2$ &  $-3$ &  $-1$ \\
  % 4761/3628800
  \cepsfbox{qcycle-429.mps} & $+2$ &  $+2$ &  $-3$ &  $-2$ &  $+6$ &  $-2$ &  $-2$ &  $-2$ &  $+1$ &  $-1$ \\
  % 5673/3628800
  \cepsfbox{qcycle-430.mps} & $+6$ &  $+2$ &  $+5$ &  $-2$ &  $-6$ &  $-2$ &  $-6$ &  $-6$ &  $-3$ &  $-1$ \\
  % 10841/3628800
  \cepsfbox{qcycle-435.mps} & $-2$ &  $+2$ &  $+1$ &  $+2$ &  $-6$ &  $+2$ &  $-2$ &  $+6$ &  $-3$ &  $+1$ \\
  % 2336/3628800
  \cepsfbox{qcycle-437.mps} & $+2$ &  $+2$ &  $+1$ &  $-2$ &  $+6$ &  $-2$ &  $-2$ &  $+6$ &  $-3$ &  $-1$ \\
  % 7506/3628800
  \cepsfbox{qcycle-438.mps} & $+6$ &  $+2$ &  $+9$ &  $-2$ &  $+2$ &  $+6$ &  $-6$ &  $+2$ &  $+1$ &  $-1$ \\
  % 15488/3628800
  \cepsfbox{qcycle-440.mps} & $+6$ &  $+2$ &  $+9$ &  $-2$ &  $+10$ &  $-2$ &  $-6$ &  $+2$ &  $+5$ &  $-1$ \\[5pt]
  % 15950/3628800
\hline
\end{tabular}
\end{center}
\caption{Orientations of a cycle of length ten that are not quasirandom-forcing by Theorem~\ref{thm:neg}.
	 For brevity, the coefficients $\alpha_C(2n_1,\ldots,2n_k)$ are represented just by $(2n_1,\ldots,2n_k)$ in the first line.}
\label{tab:10:neg}
\end{table}

Applying Theorem~\ref{thm:negsum} requires choosing an antisymmetric kernel $U$.
Two of the kernels that the theorem is applied with, the kernels $U_2$ and $U_T$,
have already been defined in Section~\ref{sec:negative} (also see Figure~\ref{fig:U23}), and
the values of $t(Q_{2n},U)$ and $t(D_{2n},U)$ for $U=U_2$ and $U=U_T$
can be found at the end of Section~\ref{sec:negative}.
In particular,
it holds that $t(D_{10},U_2)=\frac{1}{2^{20}}$ and $t(D_{10},U_T)=\frac{31}{4 \cdot 9!}$.

We next define the remaining kernel that appears in Table~\ref{tab:10:negsum}.
Let $U_4$ be the antisymmetric kernel, visualized in Figure~\ref{fig:U4}, defined as follows:
\[U_4(x,y)=\begin{cases}
           1/2 & \mbox{if $(x,y)\in (0,1/4)\times (1/2,1)$,}\\
           1/2 & \mbox{if $(x,y)\in (1/4,1/2)\times (0,1/4)$,}\\
           1/2 & \mbox{if $(x,y)\in (1/2,1)\times (1/4,1/2)$,}\\
           -1/2 & \mbox{if $(x,y)\in (0,1/4)\times (1/4,1/2)$,}\\
           -1/2 & \mbox{if $(x,y)\in (1/4,1/2)\times (1/2,1)$,}\\
           -1/2 & \mbox{if $(x,y)\in (1/2,1)\times (0,1/4)$, and}\\
	   0 & \mbox{otherwise.}
           \end{cases}\]

\begin{table}
\renewcommand{\arraystretch}{1.2}
\begin{center}
\tabcolsep=0.12cm
\begin{tabular}{|c|ccccccccc|c|cc|}
\hline
$C$ & $(2)$ & $(4)$ & $(2,2)$ & $(6)$ & $(2,4)$ & $(2,2,2)$ & $(8)$ & $(2,6)$ & $(4,4)$ & $\gamma_C$ & kernel & sum \\
\hline
\multirow{2}{*}{\cepsfbox{qcycle-411.mps}} & \multirow{2}{*}{$+2$} & \multirow{2}{*}{$+2$} & \multirow{2}{*}{$-11$} & \multirow{2}{*}{$+2$} & \multirow{2}{*}{$-10$} & \multirow{2}{*}{$-6$} & \multirow{2}{*}{$+2$} & \multirow{2}{*}{$+2$} & \multirow{2}{*}{$+1$} & \multirow{2}{*}{$+1$} & $U_2$ & $-\frac{133}{2^{19}}$ \\
 & & & & & & & & & & & $U_T$ & $-\frac{1147}{1575\cdot 2^{10}}$ \\
  % 963/3628800
  \cepsfbox{qcycle-417.mps} & $+2$ &  $-2$ &  $-3$ &  $-2$ &  $-6$ &  $-2$ &  $+2$ &  $-2$ &  $-3$ &  $+1$ & $U_T$ & $-\frac{71}{175\cdot 2^{10}}$ \\
  % 2106/3628800
  \cepsfbox{qcycle-424.mps} & $+2$ &  $-6$ &  $+1$ &  $-6$ &  $-2$ &  $+2$ &  $+2$ &  $+2$ &  $+1$ &  $+1$ & $U_T$ & $-\frac{163}{675\cdot 2^{10}}$ \\
  % 2688/3628800
  \cepsfbox{qcycle-433.mps} & $+2$ &  $-6$ &  $+1$ &  $-6$ &  $+6$ &  $-6$ &  $+2$ &  $+2$ &  $+5$ &  $+1$ & $U_T$ & $-\frac{1}{9\cdot 2^{10}}$ \\
  % 3150/3628800
  \cepsfbox{qcycle-426.mps} & $+2$ &  $-6$ &  $+5$ &  $-6$ &  $-2$ &  $+2$ &  $+2$ &  $-6$ &  $+1$ &  $+1$ & $U_4$ & $-\frac{1405}{2^{29}}$ \\[5pt]
  % 3753/3628800
\hline
\end{tabular}
\end{center}
\caption{Orientations of the cycle of length ten that are not quasirandom-forcing by Theorem~\ref{thm:negsum};
         the choice of the kernel $U$ (we give two possible choices for the first orientation) and
	 the value of the corresponding sum from the statement of the theorem are given in the last column.
         For brevity, the coefficients $\alpha_C(2n_1,\ldots,2n_k)$ are represented just by $(2n_1,\ldots,2n_k)$ in the first line.}
\label{tab:10:negsum}
\end{table}

\begin{figure}
\begin{center}
\epsfbox{qcycle-3.mps}
\end{center}
\caption{The kernel $U_4$.}
\label{fig:U4}
\end{figure}

We note that the densities of $Q_2$, $Q_4$, $Q_6$, $Q_8$ and $D_{10}$ in the kernel $U_4$
are as follows:
\[t(Q_2,U_4)=\frac{1}{2^{7}}\mbox{, }
  t(Q_4,U_4)=\frac{5}{2^{13}}\mbox{, }
  t(Q_6,U_4)=\frac{5^2}{2^{19}}\mbox{, }
  t(Q_8,U_4)=\frac{5^3}{2^{25}}\mbox{, and }
  t(D_{10},U_4)=\frac{5^5}{2^{29}}\mbox{.}\]
  
We are now left with six orientations of the cycle of length ten to analyze.
Two of them are quasirandom-forcing by Theorems~\ref{thm:QR22general} and~\ref{thm:QR4general}, respectively;
these two orientations are listed in Table~\ref{tab:10:QRxgeneral}.
The rest of the section is devoted to showing that each of the remaining $4$ orientations is also quasirandom-forcing.
These $4$ orientations can be found in Table~\ref{tab:10},
where we also list the theorems that imply that they are quasirandom-forcing.

We start with the cyclic orientation of a cycle of length ten.
While the results proven in~\cite{GrzKLV23} imply that
this orientation is quasirandom-forcing,
we include a short proof based on the methods used in this paper for completeness.

\begin{table}
\begin{center}
\begin{tabular}{|c|ccccccccc|c|c|}
\hline
$C$ & $(2)$ & $(4)$ & $(2,2)$ & $(6)$ & $(2,4)$ & $(2,2,2)$ & $(8)$ & $(2,6)$ & $(4,4)$ & $\gamma_C$ & sums \\
\hline
  \multirow{2}{*}{\cepsfbox{qcycle-404.mps}} & \multirow{2}{*}{$-2$} &  \multirow{2}{*}{$+2$} &  \multirow{2}{*}{$-11$} &  \multirow{2}{*}{$-2$} &  \multirow{2}{*}{$+10$} &  \multirow{2}{*}{$+6$} &  \multirow{2}{*}{$+2$} &  \multirow{2}{*}{$+2$} &  \multirow{2}{*}{$+1$} &  \multirow{2}{*}{$-1$} & $-0.00413$ \\
  & & & & & & & & & & & $-0.06960$ \\
  \multirow{2}{*}{\cepsfbox{qcycle-425.mps}} &  \multirow{2}{*}{$-2$}  &  \multirow{2}{*}{$-6$}  &  \multirow{2}{*}{$+5$}  &  \multirow{2}{*}{$+6$}  &  \multirow{2}{*}{$+2$}  &  \multirow{2}{*}{$-2$}  &  \multirow{2}{*}{$+2$}  &  \multirow{2}{*}{$-6$}  &  \multirow{2}{*}{$+1$}  &  \multirow{2}{*}{$-1$}  & $-0.00130$ \\
  & & & & & & & & & & & $-0.03809$ \\
\hline
\end{tabular}
\end{center}
\caption{Orientations of a cycle of length ten that are quasirandom-forcing by Theorems~\ref{thm:QR22general} and \ref{thm:QR4general}, respectively;
         the approximates value of the two sums from the statement of the theorems are given in the last column.
	 For brevity, the coefficients $\alpha_C(2n_1,\ldots,2n_k)$ are represented just by $(2n_1,\ldots,2n_k)$ in the first line.}
\label{tab:10:QRxgeneral}
\end{table}

\begin{table}[ht]
\begin{center}
\tabcolsep=0.12cm
\begin{tabular}{|c|ccccccccc|c|l|}
\hline
$C$ & $(2)$ & $(4)$ & $(2,2)$ & $(6)$ & $(2,4)$ & $(2,2,2)$ & $(8)$ & $(2,6)$ & $(4,4)$ & $\gamma_C$ & \\
\hline
  \cepsfbox{qcycle-401.mps} & $-10$ &  $+10$ &  $+25$ &  $-10$ &  $-30$ &  $-10$ &  $+10$ &  $+10$ &  $+5$ &  $-1$ & by Theorem~\ref{thm:10:401} \\
  % 0/3628800
  \cepsfbox{qcycle-415.mps} & $-2$ &  $-6$ &  $+1$ &  $+6$ &  $-6$ &  $+6$ &  $+2$ &  $+2$ &  $+5$ &  $-1$ & by Theorem~\ref{thm:10:415} \\
  % 250/3628800
  \cepsfbox{qcycle-422.mps} & $-2$ &  $+6$ &  $-3$ &  $-6$ &  $-2$ &  $+2$ &  $+2$ &  $-2$ &  $+1$ &  $-1$ & by Theorem~\ref{thm:10:422} \\
  % 1648/3628800
  \cepsfbox{qcycle-431.mps} & $+2$ &  $-2$ &  $-3$ &  $-2$ &  $-6$ &  $+6$ &  $+2$ &  $+6$ &  $+1$ &  $+1$ & by Theorem~\ref{thm:10:431} \\[5pt]
  % 3918/3628800
\hline
\end{tabular}
\end{center}
\caption{Orientations of the cycle of length ten that are quasirandom-forcing
         but this is not implied by Theorem~\ref{thm:QRgeneral}.
	 The last column contains references to theorems
	 where we establish that the orientations are quasirandom-forcing.
	 For brevity, the coefficients $\alpha_C(2n_1,\ldots,2n_k)$ are represented just by $(2n_1,\ldots,2n_k)$ in the first line.}
\label{tab:10}
\end{table}

\begin{theorem}
\label{thm:10:401}
The orientation \epsfxsize=2.5ex \epsfbox{qcycle-401.mps} of a cycle of length $10$,
i.e., the cyclic orientation of a cycle of length $10$,
is quasirandom-forcing.
\end{theorem}

\begin{proof}
Let $C$ be the cyclic orientation of a cycle of length $10$.
The values of $\alpha_C(2n_1,\ldots,2n_k)$ for feasible choices of $n_1,\ldots,n_k$ can be found in Table~\ref{tab:10}.
Consider an antisymmetric kernel $U:[0,1]^2\to [-1/2,1/2]$.
We will show $t(C,1/2+U)\le 2^{-10}$ and the equality holds if and only if $U\equiv 0$.

We analyze the expression \eqref{eq:even} and start by showing that
the sum of the three terms with coefficients $\alpha_C(6)$, $\alpha_C(8)$ and $\alpha_C(2,6)$ is non-negative.
Recall that $t(Q_2,U)\le 1/12$ by Lemma~\ref{lm:Q2}.
Using Lemma~\ref{lm:positive} we obtain the following estimate:
\begin{align*}
& \frac{\alpha_C(6)t(Q_6,U)}{16}+\frac{\alpha_C(8)t(Q_8,U)}{4}+\frac{\alpha_C(2,6)t(Q_2,U)t(Q_6,U)}{4} \\
=\; & \frac{-10t(Q_6,U)}{16}+\frac{10t(Q_8,U)}{4}+\frac{10t(Q_2,U)t(Q_6,U)}{4} \\
\le\; & \frac{-5t(Q_6,U)}{8}+\frac{5t(Q_6,U)}{2\pi^2}+\frac{5t(Q_6,U)}{24} \\
=\; & \left(-\frac{1}{4}+\frac{1}{\pi^2}+\frac{1}{12}\right)\frac{5t(Q_6,U)}{2}\le 0.
\end{align*}
Using Lemmas~\ref{lm:positive} and~\ref{lm:Q2},
we obtain the following estimate on the sum of the three terms with coefficients $\alpha_C(4)$, $\alpha_C(2,4)$ and $\alpha_C(2,8)$:
\begin{align*}
& \frac{\alpha_C(4)t(Q_4,U)}{64}+\frac{\alpha_C(2,4)t(Q_2,U)t(Q_4,U)}{16}+\frac{\alpha_C(4,4)t(Q_4,U)^2}{4} \\
=\; & \frac{10t(Q_4,U)}{64}-\frac{30t(Q_2,U)t(Q_4,U)}{16}+\frac{5t(Q_4,U)^2}{4} \\
=\; & \left(\frac{1}{2}-6t(Q_2,U)\right)\frac{5t(Q_4,U)}{16}+\frac{5t(Q_4,U)^2}{4} \\
\le\; & \left(\frac{1}{2}-6t(Q_2,U)\right)\frac{5t(Q_2,U)}{16\pi^2}+\frac{5t(Q_2,U)^2}{4\pi^4} \\
=\; & \frac{5t(Q_2,U)}{32\pi^2}-\frac{5(6\pi^2-4)t(Q_2,U)^2}{16\pi^4}.
\end{align*}
The identity \eqref{eq:even} and the just established two estimates yield that
$t(C,1/2+U)$ is at most
\begin{align*}
& \frac{1}{1024}+\frac{\alpha_C(2)t(Q_2,U)}{256}+\frac{\alpha_C(2,2)t(Q_2,U)^2}{64}+\frac{\alpha_C(2,2,2)t(Q_2,U)^3}{16}\\
& + \frac{5t(Q_2,U)}{32\pi^2}-\frac{5(6\pi^2-4)t(Q_2,U)^2}{16\pi^4}
  + \gamma_C t(D_{10},U) \\
=\; & \frac{1}{1024}-\frac{10t(Q_2,U)}{256}+\frac{25t(Q_2,U)^2}{64}-\frac{10t(Q_2,U)^3}{16}\\
& + \frac{5t(Q_2,U)}{32\pi^2}-\frac{5(6\pi^2-4)t(Q_2,U)^2}{16\pi^4}
  - t(D_{10},U) \\
\le\; & \frac{1}{1024}-\frac{10t(Q_2,U)}{256}+\frac{25t(Q_2,U)^2}{64} 
  + \frac{5t(Q_2,U)}{32\pi^2}-\frac{5(6\pi^2-4)t(Q_2,U)^2}{16\pi^4}
  - t(D_{10},U) \\
=\; & \frac{1}{1024}-\frac{5(\pi^2-4)t(Q_2,U)}{128\pi^2}+\frac{5(5\pi^4-24\pi^2+16)t(Q_2,U)^2}{64\pi^4}
  - t(D_{10},U) \\
=\; & \frac{1}{1024}-\frac{5(5\pi^2-4)(\pi^2-4)}{64\pi^4}\left(\frac{\pi^2}{10\pi^2-8}-t(Q_2,U)\right)t(Q_2,U)
  - t(D_{10},U).
\end{align*}
Since $\frac{\pi^2}{10\pi^2-8}>\frac{1}{12}$ and $t(Q_2,U)\le \frac{1}{12}$ by Lemma~\ref{lm:Q2},
we have $t(C,1/2+U)\le 2^{-10}-t(D_{10},U)$.
Hence, Lemma~\ref{lm:cycle} yields that
$t(C,1/2+U)\le 2^{-10}$ and the equality holds if and only if $U\equiv 0$,
which implies that the orientation $C$ is quasirandom-forcing.
\end{proof}

We next deal with the second orientation that is not covered by the results presented in Section~\ref{sec:general}.

\begin{theorem}
\label{thm:10:415}
The orientation \epsfxsize=2.5ex \epsfbox{qcycle-415.mps} of a cycle of length $10$ is quasirandom-forcing.
\end{theorem}

\begin{proof}
Let $C$ be the orientation of a cycle from the statement of the theorem.
The values of $\alpha_C(2n_1,\ldots,2n_k)$ for feasible choices of $n_1,\ldots,n_k$ are given in Table~\ref{tab:10}.
Consider an antisymmetric kernel $U:[0,1]^2\to [-1/2,1/2]$.
We will show $t(C,1/2+U)\le 2^{-10}$ and the equality holds if and only if $U\equiv 0$.

We again analyze the expression \eqref{eq:even}.
We first estimate the sum of the two terms with the coefficients $\alpha_C(2,2)$ and $\alpha_C(2,2,2)$
using that $t(Q_2,U)\le 1/12$, which holds by Lemma~\ref{lm:Q2}:
\begin{align*}
\frac{\alpha_C(2,2)t(Q_2,U)^2}{64}+\frac{\alpha_C(2,2,2)t(Q_2,U)^3}{16}
& = \frac{t(Q_2,U)^2}{64}+\frac{6t(Q_2,U)^3}{16} \\
& \le \left(\frac{1}{64\cdot 12}+\frac{6}{16\cdot 12^2}\right) t(Q_2,U) \\
& = \frac{t(Q_2,U)}{256}.
\end{align*}
We next estimate the sum of the four terms with the coefficients $\alpha_C(6)$, $\alpha_C(8)$,
$\alpha_C(2,6)$ and $\alpha_C(4,4)$ using Lemmas~\ref{lm:positive} and~\ref{lm:Q2}:
\begin{align*}
& \frac{\alpha_C(6)t(Q_6,U)}{16}+\frac{\alpha_C(8)t(Q_8,U)}{4}+\frac{\alpha_C(2,6)t(Q_2,U)t(Q_6,U)}{4}+\frac{\alpha_C(4,4)t(Q_4,U)^2}{4}\\
=\; & \frac{6t(Q_6,U)}{16}+\frac{2t(Q_8,U)}{4}+\frac{2t(Q_2,U)t(Q_6,U)}{4}+\frac{5t(Q_4,U)^2}{4} \\
\le\; & \left(\frac{6}{16\pi^2}+\frac{2}{4\pi^4}+\frac{2}{4\cdot 12\pi^2}+\frac{5}{4\cdot 12\pi^2}\right)t(Q_4,U)
=   \frac{24+25\pi^2}{48\pi^4} t(Q_4,U).
\end{align*}
Hence, the identity \eqref{eq:even} implies that $t(C,1/2+U)$ is at most
\begin{align*}
& \frac{1}{1024} + \frac{\alpha_C(2)t(Q_2,U)}{256} + \frac{t(Q_2,U)}{256} + \frac{\alpha_C(4)t(Q_4,U)}{64} + \frac{(24+25\pi^2)t(Q_4,U)}{48\pi^4} + \gamma_C t(D_{10},U) \\
& = \frac{1}{1024} - \frac{2t(Q_2,U)}{256} + \frac{t(Q_2,U)}{256} - \frac{6t(Q_4,U)}{64} + \frac{(24+25\pi^2)t(Q_4,U)}{48\pi^4} - t(D_{10},U) \\
& \le \frac{1}{1024} - t(D_{10},U).
\end{align*}
Using Lemma~\ref{lm:cycle},
we conclude that $t(C,1/2+U)\le 2^{-10}$ and the equality holds if and only if $U\equiv 0$,
which yields that $C$ is quasirandom-forcing.
\end{proof}

We now deal with the third orientation not covered by the results from Section~\ref{sec:general}.

\begin{theorem}
\label{thm:10:422}
The orientation \epsfxsize=2.5ex \epsfbox{qcycle-422.mps} of a cycle of length $10$ is quasirandom-forcing.
\end{theorem}

\begin{proof}
Let $C$ be the orientation of a cycle from the statement of the theorem.
The values of $\alpha_C(2n_1,\ldots,2n_k)$ for feasible choices of $n_1,\ldots,n_k$ are given in Table~\ref{tab:10}.
Consider an antisymmetric kernel $U:[0,1]^2\to [-1/2,1/2]$.
We will show $t(C,1/2+U)\le 2^{-10}$ and the equality holds if and only if $U\equiv 0$.

Again, as in the proofs of Theorems~\ref{thm:10:401} and~\ref{thm:10:415},
we analyze the expression \eqref{eq:even}.
Using $t(Q_2,U)\le 1/12$, which holds by Lemma~\ref{lm:Q2},
we show that the sum of the two terms with the coefficients $\alpha_C(2,2)$ and $\alpha_C(2,2,2)$ is non-positive:
\begin{align*}
\frac{\alpha_C(2,2)t(Q_2,U)^2}{64}+\frac{\alpha_C(2,2,2)t(Q_2,U)^3}{16}
& = \frac{-3t(Q_2,U)^2}{64}+\frac{2t(Q_2,U)^3}{16} \\
& \le \frac{-3t(Q_2,U)^2}{64}+\frac{t(Q_2,U)^2}{96} \le 0.
\end{align*}
Using Lemmas~\ref{lm:positive} and~\ref{lm:Q2},
we show that the sum of the two terms with the coefficients $\alpha_C(2,4)$ and $\alpha_C(4,4)$ is also non-positive:
\begin{align*}
\frac{\alpha_C(2,4)t(Q_2,U)t(Q_4,U)}{16}+\frac{\alpha_C(4,4)t(Q_4,U)^2}{4}
& = \frac{-2t(Q_2,U)t(Q_4,U)}{16}+\frac{t(Q_4,U)^2}{4} \\
& \le \frac{-t(Q_2,U)t(Q_4,U)}{8}+\frac{t(Q_2,U)t(Q_4,U)}{4\pi^2} \\
& \le 0.
\end{align*}
Since $\alpha_C(2,6)<0$, it follows from \eqref{eq:even} that $t(C,1/2+U)$ is at most
\begin{align*}
&\frac{1}{1024}
    +\frac{\alpha_C(2)t(Q_2,U)}{256}
    +\frac{\alpha_C(4)t(Q_4,U)}{64}
    +\frac{\alpha_C(6)t(Q_6,U)}{16}
    +\frac{\alpha_C(8)t(Q_8,U)}{4}
	+\gamma_C t(D_{10},U) \\
& = \frac{1}{1024}
    -\frac{2t(Q_2,U)}{256}
    +\frac{6t(Q_4,U)}{64}
    -\frac{6t(Q_6,U)}{16}
    +\frac{2t(Q_8,U)}{4}
    -t(D_{10},U) \\
& = \frac{1}{1024}
    -\frac{t(Q_2,U)}{128}
    +\frac{3t(Q_4,U)}{32}
    -\frac{3t(Q_6,U)}{8}
    +\frac{t(Q_8,U)}{2}
    -t(D_{10},U). 
\end{align*}    
Note that
\[ 
\frac{x}{128}-\frac{3x^2}{32}+\frac{3x^3}{8}-\frac{x^4}{2}=\frac{x}{2}\left(\frac{1}{4}-x\right)^3
\]
is non-negative for every $x\in [0,1/\pi^2]$.
Hence, Lemma~\ref{lm:polynomial} yields that
\[\frac{t(Q_2,U)}{128}-\frac{3t(Q_4,U)}{32}+\frac{3t(Q_6,U)}{8}-\frac{t(Q_8,U)}{2}\ge 0,\]
which implies that
\[t(C,1/2+U) \le \frac{1}{1024}-t(D_{10},U).\]
It now follows from Lemma~\ref{lm:cycle} that
$t(C,1/2+U)\le 2^{-10}$ and the equality holds if and only if $U\equiv 0$,
which yields that $C$ is quasirandom-forcing.
\end{proof}

The last remaining orientation not covered by the results from Section~\ref{sec:general}
is dealt with in the next theorem.

\begin{theorem}
\label{thm:10:431}
The orientation \epsfxsize=2.5ex \epsfbox{qcycle-431.mps} of a cycle of length $10$ is quasirandom-forcing.
\end{theorem}

\begin{proof}
Let $C$ be the orientation of a cycle from the statement of the theorem.
The values of $\alpha_C(2n_1,\ldots,2n_k)$ for feasible choices of $n_1,\ldots,n_k$ are given in Table~\ref{tab:10}.
Consider an antisymmetric kernel $U:[0,1]^2\to [-1/2,1/2]$.
We will show $t(C,1/2+U)\ge 2^{-10}$ and the equality holds if and only if $U\equiv 0$.

As in the proofs of Theorems~\ref{thm:10:401}, \ref{thm:10:415} and~\ref{thm:10:422},
we analyze the expression \eqref{eq:even}.
Using $t(Q_2,U)\le 1/12$, which holds by Lemma~\ref{lm:Q2},
we first show that
the sum of a part of the term with coefficient $\alpha_C(2)$, and
the two terms with the coefficients $\alpha_C(2,2)$ and $\alpha_C(2,2,2)$ is non-negative:
\begin{align*}
& \frac{6t(Q_2,U)}{64^2}+\frac{\alpha_C(2,2)t(Q_2,U)^2}{64}+\frac{\alpha_C(2,2,2)t(Q_2,U)^3}{16}\\
=\; & \frac{6t(Q_2,U)}{64^2}-\frac{3t(Q_2,U)^2}{64}+\frac{6t(Q_2,U)^3}{16}
    = 6t(Q_2,U)\left(\frac{1}{64}-\frac{t(Q_2,U)}{4}\right)^2 \ge 0.
\end{align*}
Since $\alpha_C(2)=2$ and $\frac{2}{256}-\frac{6}{64^2}=\frac{13}{2048}$ (and
$\alpha_C(4,4)$ and $\alpha_C(8)$ are positive),
we obtain using the identity \eqref{eq:even} that $t(C,1/2+U)$ is at least
\begin{align*}
& \frac{1}{1024}+\frac{13t(Q_2,U)}{2048}+\frac{\alpha_C(4)t(Q_4,U)}{64}+\frac{\alpha_C(6)t(Q_6,U)}{16}
                +\frac{\alpha_C(2,4)t(Q_2,U)t(Q_4,U)}{16}\\
	      & + \frac{\alpha_C(2,6)t(Q_2,U)t(Q_6,U)}{4}
	        + \gamma_C t(D_{10},U)\\
=\; & \frac{1}{1024}+\frac{13t(Q_2,U)}{2048}-\frac{2t(Q_4,U)}{64}-\frac{2t(Q_6,U)}{16}
                  -\frac{6t(Q_2,U)t(Q_4,U)}{16}\\
		 & + \frac{6t(Q_2,U)t(Q_6,U)}{4}
		  + t(D_{10},U) \\
=\; & \frac{1}{1024}+\frac{13t(Q_2,U)}{2048}-\frac{t(Q_4,U)}{32}-\frac{t(Q_6,U)}{8}
                  -\frac{3t(Q_2,U)t(Q_4,U)}{8}\\
		 & + \frac{3t(Q_2,U)t(Q_6,U)}{2}
		  + t(D_{10},U).
\end{align*}
Since $t(Q_2,U)\le 1/12$ by Lemma~\ref{lm:Q2},
it holds that $\frac{1}{8}-\frac{3t(Q_2,U)}{2}\ge 0$ and
so we can obtain using Lemma~\ref{lm:positive} that $t(C,1/2+U)$ is at least
\begin{align*}
& \frac{1}{1024}+\frac{13t(Q_2,U)}{2048}-\frac{t(Q_4,U)}{32}-\frac{3t(Q_2,U)t(Q_4,U)}{8}
                -\left(\frac{1}{8}-\frac{3t(Q_2,U)}{2}\right)t(Q_6,U) \\
		& + t(D_{10},U)\\
\ge\; & \frac{1}{1024}+\frac{13t(Q_2,U)}{2048}-\frac{t(Q_4,U)}{32}-\frac{3t(Q_2,U)t(Q_4,U)}{8}
                -\left(\frac{1}{8}-\frac{3t(Q_2,U)}{2}\right)\frac{t(Q_4,U)}{\pi^2}\\
		& + t(D_{10},U)\\
=\; & \frac{1}{1024}+\frac{13t(Q_2,U)}{2048}-\frac{(4+\pi^2)t(Q_4,U)}{32\pi^2}
                    -\frac{3(\pi^2-4)t(Q_2,U)t(Q_4,U)}{8\pi^2} + t(D_{10},U) \\
\ge\; & \frac{1}{1024}+\frac{13t(Q_2,U)}{2048}-\frac{(4+\pi^2)t(Q_2,U)}{32\pi^4}
                    -\frac{(\pi^2-4)t(Q_2,U)}{32\pi^4} + t(D_{10},U) \\
\ge\; & \frac{1}{1024}+\left(\frac{13}{2048}-\frac{1}{16\pi^2}\right)t(Q_2,U)+t(D_{10},U)
        \ge \frac{1}{1024}+t(D_{10},U).
\end{align*}
Since we have derived that $t(C,1/2+U)$ is at least $2^{-10}+t(D_{10},U)$,
we can conclude using Lemma~\ref{lm:cycle} that $t(C,1/2+U)=2^{-10}$ if and only if $U\equiv 0$,
i.e., $C$ is quasirandom-forcing.
\end{proof}

\section{Concluding remarks}
\label{sec:concl}

We finish with a brief discussion on our results.
We start with noting that the proofs of all our results that
an orientation of a cycle is quasirandom-forcing are based on spectral arguments.
While there exist non-spectral arguments for some of the orientations,
e.g., the last two orientations of the cycle of length four given in Table~\ref{tab:4},
we believe that this may be a general phenomenon,
i.e., that if an orientation of a cycle is quasirandom-forcing,
then there exists a proof based on spectral arguments.

It is an intriguing question to find (if one exists) a simple description
which orientations of cycles are quasirandom-forcing and which are not.
Perhaps the most natural idea is to attempt finding a classification
based on the number of forward and backward edges of an orientation;
this is however impossible as
there is an orientation of the cycle of length six with four forward and two backward edges that is quasirandom-forcing and
there is also such an orientation that is not quasirandom-forcing (see Table~\ref{tab:6}).

While we have phrased our arguments showing that
a certain orientation of a cycle is not quasirandom-forcing in terms of kernels,
all our arguments are actually constructive in the sense that
they give a description of a tournamenton witnessing that the orientation is not quasirandom-forcing.
We have attempted to provide the simplest possible arguments,
however,
it is worth noting that the following holds for orientations of cycles of length four, six and eight:
\emph{an orientation $C$ of an even cycle of length at most eight is quasirandom-forcing
      if and only if $\gamma_C$ and $t(C,1/2+U_T)-2^{-\|C\|}$ have the same sign}.
Indeed,
if an orientation $C$ is quasirandom-forcing,
then $\gamma_C$ and $t(C,1/2+U_T)-2^{-\|C\|}$ must have the same sign by Theorem~\ref{thm:transitive}.
The opposite implication,
i.e., if $C$ is not quasirandom-forcing,
then $\gamma_C$ and $t(C,1/2+U_T)-2^{-\|C\|}$ have different signs,
is demonstrated in Table~\ref{tab:trans468},
where all orientations of even cycles of length at most eight that
are not quasirandom-forcing are listed together with the values of $\gamma_C$ and $t(C,1/2+U_T)-2^{-\|C\|}$.

\begin{table}
\renewcommand{\arraystretch}{1.4}
\begin{center}
\begin{tabular}{|c|c|cccc|}
\hline
$C$ & \cepsfbox{qcycle-101.mps} &
      \cepsfbox{qcycle-202.mps} & \cepsfbox{qcycle-204.mps} & \cepsfbox{qcycle-205.mps} & \cepsfbox{qcycle-206.mps} \\[3pt]
\hline
$\gamma_C$ & $+1$ & $+1$ & $-1$ & $+1$ & $-1$ \\
\hline
$t(C,1/2+U_T)$ & $\frac{0}{24}$ & $\frac{1}{720}$ & $\frac{13}{720}$ & $\frac{6}{720}$ & $\frac{18}{720}$ \\
$t(C,1/2+U_T)-2^{-\|C\|}$ &
                 $\frac{-1.5}{24}$ & $\frac{-10.25}{720}$ & $\frac{1.75}{720}$ & $\frac{-5.25}{720}$ & $\frac{6.75}{720}$ \\[3pt]
\hline
\end{tabular}
\vskip 2ex
\tabcolsep=0.12cm
\begin{tabular}{|c|ccccccccc|}
\hline
$C$ & \cepsfbox{qcycle-301.mps} & \cepsfbox{qcycle-303.mps} & \cepsfbox{qcycle-304.mps} & \cepsfbox{qcycle-306.mps} &
      \cepsfbox{qcycle-308.mps} & \cepsfbox{qcycle-309.mps} & \cepsfbox{qcycle-311.mps} & \cepsfbox{qcycle-312.mps} &
      \cepsfbox{qcycle-314.mps} \\[3pt]
\hline
$\gamma_C$ & $+1$ & $+1$ & $+1$ & $+1$ & $+1$ & $+1$ & $-1$ & $+1$ & $-1$ \\
\hline
$t(C,1/2+U_T)$ & $\frac{0}{40320}$ & $\frac{6}{40320}$ & $\frac{26}{40320}$ &
                 $\frac{54}{40320}$ & $\frac{20}{40320}$ & $\frac{68}{40320}$ &
                 $\frac{297}{40320}$ & $\frac{110}{40320}$ & $\frac{423}{40320}$ \\
$t(C,1/2+U_T)-2^{-\|C\|}$ &
                 $\frac{-157.5}{40320}$ & $\frac{-151.5}{40320}$ & $\frac{-131.5}{40320}$ &
                 $\frac{-103.5}{40320}$ & $\frac{-137.5}{40320}$ & $\frac{-89.5}{40320}$ &
                 $\frac{139.5}{40320}$ & $\frac{-47.5}{40320}$ & $\frac{265.5}{40320}$ \\[3pt]
\hline		 
\end{tabular}
\end{center}
\caption{The numerical values demonstrating that an orientation $C$ of an even cycle of length at most eight
         is quasirandom-forcing if and only if $\gamma_C$ and $t(C,1/2+U_T)-2^{-\|C\|}$ have the same sign.}
\label{tab:trans468}
\end{table}

%If the equivalence were true in general, it would imply that
%an orientation $C$ of an even cycle of length is quasirandom-forcing if and only if
%either $t(C,1/2+U_T)>2^{-\|C\|}$ and $t(C,1/2+U_3)>2^{-\|C\|}$ or
%       $t(C,1/2+U_T)<2^{-\|C\|}$ and $t(C,1/2+U_3)<2^{-\|C\|}$.
%Unfortunately, the equivalence holds for all but one orientations of the cycle of length ten.
%The only orientation of a cycle of length ten that fails to satisfy the equivalence
%is the orientation given on the last line in Table~\ref{tab:10}:
Unfortunately, such equivalence is not true in general, as it fails for one orientation of a cycle of length ten:
the orientation given on the last line in Table~\ref{tab:10}.
For this orientation~$C$,
it holds that $\gamma_C=+1$ but $t(C,1/2+U_T)=3753/3628800>2^{-10}$.
So, this attempt also fails.

Perhaps, the following (likely much easier) problem could help with understanding the boundary
between orientations that are quasirandom-forcing and that are not.

\begin{problem}
Let $p_{2n}$ be the probability that the orientation of a cycle of length $2n$, $n\ge 2$,
obtained by orienting each edge randomly in each direction with probability $1/2$ independently of other edges
is quasirandom-forcing.
Is the sequence $(p_{2n})_{n\in\NN}$ converging? If so, what is its limit?
\end{problem}

We remark that our classification results imply $p_4=7/8=0.875$,
$p_6=7/16\approx 0.438$, $p_8=71/128\approx 0.555$ and $p_{10}=207/512\approx 0.404$.

\bibliographystyle{bibstyle}
\bibliography{qcycle}

\end{document}